\documentclass{amsart}
\usepackage{graphicx}
\usepackage{graphicx}
\usepackage{color}
\usepackage{times}
\usepackage{cite}
\usepackage{url}
\usepackage{enumerate,latexsym}
\usepackage{latexsym}
\usepackage{amsmath,amssymb}
\usepackage{graphicx}
\usepackage{amsthm}
\usepackage{verbatim}
\newtheorem{thm}{Theorem}[section]
\newtheorem*{theorem*}{Theorem}
\newtheorem*{acknowledgement*}{Acknowledgement}

\newtheorem{lem}[thm]{Lemma}
\newtheorem{prop}[thm]{Proposition}
\theoremstyle{definition}

\theoremstyle{remark}
\newtheorem{rem}[thm]{Remark}
\newtheorem{ques}[thm]{Question}
\newtheorem{conj}[thm]{Conjecture}
\numberwithin{equation}{section}

\newcommand{\set}[1]{\left\{#1\right\}}
\newcommand{\Real}{\mathbb R}

\newcommand{\func}[1]{\ensuremath{\mathop{\mathrm{#1}}} }

\newcommand{\spt}[0]{\func{spt}}

\begin{document}
	\title{Minimal equatorial fillings of low index}
	\author{Jacob Bernstein}\address{Johns Hopkins University \\ 3400 N. Charles St, Baltimore MD 21218 USA}
	\email{jberns15@jhu.edu}
	\author{Daniel Ketover}\address{Rutgers University\\  Busch Campus - Hill Center \\ 110 Freylinghausen Road, Piscataway NJ 08854 USA}
	
	\email{dk927@math.rutgers.edu}
	\begin{abstract}

     We obtain restrictions on the geometry and topology of minimal equatorial fillings in $\mathbb{S}^3$ of the lowest possible Morse index.
	\end{abstract}	

	\maketitle
\section{Introduction}
Consider the great circle in the unit three-sphere
$$
C=\set{x_3=x_4=0}\cap \mathbb{S}^3\subset \mathbb{S}^3 \subset \mathbb{R}^4.
$$  
We study the compact minimal surfaces, $\Sigma\subset \mathbb{S}^3$ with $\partial \Sigma=C$ that have low Morse index with respect to variations fixing the boundary circle, i.e., minimal surfaces with few negative Dirichlet eigensections of the stability operator.  We are principally interested in embedded minimal surfaces, i.e., submanifolds.  However, at times it will be convenient to consider minimal immersions $\iota:N\to \Sigma\subset \mathbb{S}^3$ of a compact surface $N$ with boundary, in this case we insist $\iota:\partial N \to C\subset \mathbb{S}^3$ is an embedding -- this means $\Sigma$ spans $C$ in an appropriate sense.  We say such an immersion is \emph{embedded near $C$} if $\iota^{-1}(C)=\partial N$.  When $\mathrm{Area}(\Sigma)<6\pi$, then $\Sigma$ is embedded -- see \cite[Lemma 2.5]{BKcan}.

When $\Sigma$ is oriented and embedded the only such surface is a hemisphere which has Morse index zero \cite{Hardt2011}.  If $\Sigma$ is non-orientable, then there is an embedded minimal M\"{o}bius band, $\Sigma_{LB}$, which we also denote by $\bar{\tau}_{1,2}$.  This surface is half of Lawson's $\tau_{1,2}$ surface from \cite{lawsonCompleteMinimalSurfaces1970}, and which we call the \emph{Lawson band}.  It is explicitly parameterized by:
$$
\mathbf{F}_{LB}(x,y)=(\cos y\cos x, \sin y \cos x, \cos 2y \sin x, \sin 2y \sin x),
$$
where here $0\leq x \leq \pi$ and $0\leq y < \pi$. The surface $\Sigma_{LB}$ has Morse index two. There is a one-parameter family such minimal M\"{o}bius bands obtained by rotating in the $x_1x_2$ plane.  There is also a ``mirror" family with opposite chirality obtained by reflecting across the $x_4=0$ plane.  Other immersed example are produced by taking half of Lawson's $\tau_{1,2k}$ surfaces for $k>1$, denote these by  $\bar{\tau}_{1,2k}$. In \cite{BKgenus}, we construct many other embedded examples of positive genus.

We make the following conjecture:
\begin{conj}\label{MainConj}
 If $\Sigma$ is a minimal surface bounded by $C$ with Morse index at most two, then either $\Sigma$ is a hemisphere or $\Sigma$ is, up to orthogonal transformations preserving $C$, $\Sigma_{LB}$.  
\end{conj}
 This is analogous to Urbano's characterization of the Clifford torus in $\mathbb{S}^3$ in terms of its index \cite{Urbano1990}.
There is a weakening of this conjecture which is still of significant interest:
\begin{conj}\label{NoFakeLBConj}
	If $\Sigma$ is a minimal M\"{obius} band bounded by $C$ with Morse index two, then $\Sigma$ is, up to orthogonal transformations preserving $C$, $\Sigma_{LB}$.  
\end{conj}
In this note we collect some restrictions on the topology and geometry of index two minimal surfaces bounded by $C$ which provide some evidence for these conjectures.

Let us first fix some terminology.  
 For a minimal surface $\Sigma$ bounded by $C$, its \emph{(Dirichlet) Morse index} is the maximal dimension of the space of deformations that fix $C$ and decrease area to second order.  We denote this by $\mathrm{ind}_D(\Sigma)$.  The \emph{(Dirichlet) nullity} of $\Sigma$ is the dimension of the Dirichlet kernel of the stability operator we denote this by $\mathrm{null}_D(\Sigma)$.  One computes -- see  Proposition \ref{LBIndNullProp} -- that
$$
\mathrm{ind}_D(\Sigma_{LB})=2 \mbox{ and } \mathrm{null}_D(\Sigma_{LB})=1.
$$

For a general non-orientable surface bounded by $C$, $\Sigma$,  its \emph{genus}, $\mathrm{g}(\Sigma)$ is the number of cross-caps making up $\Sigma$.  For instance, if $\Sigma$ is a M\"obius band, then $\mathrm{g}(\Sigma)=1$.  Similarly, if we pick an orientation on $C$, then we define the Euler class or winding number, $e(\Sigma)$, to be the linking number of a small push out by the  outward conormal, $\nu$, to $\Sigma$ of $C$ with $C$ itself.  If $\Sigma$ is embedded near $C$, then $e(\Sigma)$ is an even integer -- see \cite[Lemma 3.2]{BKcan}.  When $\Sigma$ is oriented and embedded, $e(\Sigma)=0$ while $e(\Sigma_{LB})=2$. In general, as explained in \cite[Section 2]{Conway} for an embedded non-orientable surface it follows from a result of  Whitney \cite{Whitney} that
$$e(\Sigma)= 2g(\Sigma) \mod  4.$$
There are stronger restrictions on $e(\Sigma)$ for such $\Sigma$ -- \cite{Ma, Yasuhara} and discussion in \cite{BKgenus}.   

If we let $s$ be the arclength coordinate on $C$, and express the outward conormal as
$$
\nu=\cos \theta(s) \mathbf{e}_3+\sin \theta(s)\mathbf{e}_4,
$$
then $\kappa(s)=\theta'(s)$ is called the \emph{winding speed} and satisfies
$$
2 \pi e(\Sigma)=\int_0^{2\pi} \kappa(s) ds.
$$
We note that the winding speed of $\Sigma_{LB}$ is 2.

We establish the following restrictions on the topology and geometry of index two minimal surfaces bounded by $C$.
\begin{thm}\label{MainThm}
	Let $\Sigma$ be a (necessarily) non-orientable compact minimal surface bounded by $C$ of Morse index two.  The following is true:
	\begin{enumerate}
		\item If $\Sigma$ is embedded, then $|e(\Sigma)|\leq g(\Sigma)+1$ while if $\Sigma$ is immersed and  embedded near $C$ and $g(\Sigma)=1$, then $|e(\Sigma)|=2$;
			\item If $g(\Sigma)=1$, then $1\leq \mathrm{null}_D(\Sigma)\leq 3$ and, moreover, $\mathrm{null}_D(\Sigma)=1$ when and only when $\Sigma$ is, up to symmetry, equal to $\Sigma_{LB}$;
		\item If $g(\Sigma)=1$, i.e., $\Sigma$ is a M\"{o}bius band, and $|\kappa|\geq 1$, then $\Sigma$ is, up to symmetry, equal to $\Sigma_{LB}$;
						\item If $\Sigma$ is embedded, then $g(\Sigma)\neq 2$, i.e., there are no embedded minimal punctured Klein bottles of index two.
	\end{enumerate}
\end{thm}
\begin{proof}
	The first claim of Item (1) is an immediate consequence of Proposition \ref{NodalCountWProp}.  To see the second claim observe that by Proposition \ref{NodalCountWProp}, when $g(\Sigma)=1$, $|e(\Sigma)|\leq 2$.  As $\Sigma$ is embedded near $C$, $e(\Sigma)$ is even.   By Proposition \ref{NoWindZeroProp}, $e(\Sigma)\neq 0$ concluding the proof.
	
Item (2) follows from Propositions \ref{JacobiNodalProp} and \ref{NullProp}.   	Item (3) is shown in Proposition \ref{NoSlowSpeedCharProp}.   Finally, Item (4) is shown in Theorem \ref{NoKBInd2Thm}.
\end{proof}

We also obtain restrictions under additional symmetry hypotheses.  To that end, let $M$ be a \emph{meridian} which is half of a great circle in $\mathbb{S}^3$ that meets $C$ orthogonally at its two end points.  When $\mathcal{R}_{M}: \mathbb{R}^4 \to \Real^4$ is the reflection across the great circle containing $M$ and $M$ is contained in a minimal surface $\Sigma$, then the Schwarz reflection principle implies $\Sigma$ is symmetric with respect to $\mathcal{R}_M$.

Consider the following meridian
$$
M_0=\set{(\cos x, 0, \sin x,0): 0\leq x \leq \pi}\subset \Sigma_{LB},
$$
with corresponding reflection
$$
\mathcal{R}_{M_0}=(x_1,x_2,x_3,x_4)=(x_1,-x_2, x_3,-x_4).
$$
This is a symmetry of $\Sigma_{LB}$ that satisfies
$$
\mathcal{R}_{M_0}\circ \mathbf{F}_{LB}(x,y)= \mathbf{F}_{LB}(x,-y).
$$
In fact, $\Sigma_{LB}$ is foliated by meridians.
In Section \ref{SymSec} we establish the following rigidity result for index two minimal M\"obius bands containing at least two meridians:
\begin{thm}\label{SymRigidThm}
 Let $\Sigma$ be a minimal M\"obius band bounded by $C$ of Morse index two.  If $\Sigma$ contains at least two distinct meridians, then, up to symmetry, $\Sigma=\Sigma_{LB}$.
\end{thm}	
In a similar vein, we observe that the family of positive genus examples, $\bar{\xi}_{m,k}$ from \cite{BKgenus} do not have index two:
\begin{thm} \label{BKIndThm}
All the minimal surfaces, $\bar{\xi}_{m,k}$, have Morse index strictly greater than two.
\end{thm}
Finally, in Appendix \ref{IntSec} we use ideas from the integrable systems literature -- specifically \cite{pinkallClassificationConstantMean1989} to obtain some additional information about the index and nullity of minimal M\"obius bands that shed further light on Conjecture \ref{NoFakeLBConj}.  In particular, any counterexample to the conjecture must lie within some finite dimensional space.
\\
\newline
\emph{Acknowledgements}:  J.B was partially supported by the NSF grant DMS-2203132.  D.K. was partially supported by NSF grant DMS-2405114.

\section{Basic Observations}
\subsection{Notation}
We suppose $\Sigma$ is a non-orientable minimal surface bounded by $C$.  We allow $\Sigma$ to be immersed which we take to mean that there is a minimal immersion of a non-orientable surface with boundary $\phi:N\to \Sigma\subset \mathbb{S}^3$ with $\phi: \partial N \to C$ an embedding. Let $\pi: \Sigma'\to \Sigma$ by the oriented double cover of $\Sigma$ and let $\zeta: {\Sigma'}\to {\Sigma'}$ be the deck transform.  If $\Sigma$ is immersed this means there is an immersion $\phi':N'\to \Sigma'$ with $\pi:N'\to N$ the orientation double cover and $\phi\circ \pi=\phi'$ and $\zeta:N'\to N'$.  We note that ${\Sigma'}$ has two boundary components, $\partial {\Sigma'}={C'}_+\cup {C'}_-$ and $\zeta(C_\pm')=C_\mp'$. 
 
 We say a function $u$ on $\Sigma'$ is \emph{even} if $u\circ \zeta=u$ and is \emph{odd} if $u\circ \zeta =-u$.   On $\Sigma'$ there is a well defined choice of unit normal $\mathbf{n}$ that satisfies $\mathbf{n}\circ \zeta=-\mathbf{n}$.    Clearly, for any odd function $u$, there is a unique section $\mathbf{s}$ of the normal bundle of $\Sigma$ so that $u \mathbf{n}=\mathbf{s}\circ \pi$. The converse is also true, for any such section, there is a unique choice of odd function on $\Sigma'$.  These concepts are adapted to the immersed setting in the obvious way.

Let us denote by $\Delta_{\Sigma}$ and $\Delta_{\Sigma'}$, the Laplace Beltrami operator of $\Sigma$ and $\Sigma'$ respectively.  We have $\Delta_{\Sigma} x_i=-2 x_i$ and on $C$
$$
\partial_\nu x_1=\partial_\nu x_2=0, x_3=x_4=0.
$$
In particular, $x_3$ and $x_4$ are Dirichlet eigenfunctions of $\Delta_\Sigma$ while 
$x_1$ and $x_2$ are Neumann eigenfunctions.  Likewise, let us denote by $L_{\Sigma}$ the stability operator on on the normal bundle of $\Sigma$ which we express as
$$
L_{\Sigma}=\Delta^\perp_{\Sigma} +|A_{\Sigma}|^2 +2
$$ 
and we note that while the second fundamental form is not globally well defined on $\Sigma$, its squared norm $|A_{\Sigma}|^2$ is.
One readily verifies that if $\mathbf{e}_i^\perp$ is the orthogonal projection of the coordinate vector fields $\mathbf{e}_i$ to the normal bundle, then
$$
-L_{\Sigma} \mathbf{e}_i^\perp=-2 \mathbf{e}_i^\perp.
$$
In order to study the specturm of the stability operator on $\Sigma$ it is convenient to consider the spectrum of the operator
$$
L_{\Sigma'}=\Delta_{\Sigma'} +|A_{\Sigma'}|^2 +2
$$
restricted to odd functions on $\Sigma'$.  The (Dirichlet) Morse index of $\Sigma$ is the dimension of the space of odd Dirichlet eigenfunctions of $-L_{\Sigma'}$ which have negative eigenvalue.  In this setting one verifies that, after choice of unit normal, $\mathbf{n}$, $n_i=\mathbf{n}\cdot \mathbf{e}_i$ are odd, satisfy 
$$
-L_{\Sigma'} n_i=-2 n_i$$
and, on $C'=\partial \Sigma'$,
$$
n_1=n_2=0 \mbox{ and } \partial_{\nu'} n_3=\partial_{\nu'} n_4=0.
$$
That is, $n_1$ and $n_2$ are odd Dirichlet eigenfunctions and $n_3$ and $n_4$ are odd Neumann eigenfunctions.  In summary, if
$$
\mathcal{U}'_-=\set{u\in C^\infty(\Sigma') : - L_{\Sigma'}u =\lambda u, u|_{\partial \Sigma'}=0, u\circ \zeta =-u, \lambda<0}
$$
be set of odd Dirichlet eigenvalues of $-L_{\Sigma'}$ with negative eigenvalue, then $n_1,n_2\in \mathcal{U}'$ and the Dirichlet Morse index of $\Sigma$, can be computed as
$$
\mathrm{ind}_D(\Sigma)=\dim \mathcal{U}'_-.
$$

In a similar fashion, on ${\Sigma'}$ we have the following six odd functions in the kernel of $L_{\Sigma'}$, for $1\leq i <j\leq 4$, 
$$R_{ij}= x_i n_j - x_j n_i.$$
 Geometrically, these correspond to the normal component of rotation in the $x_i x_j$-plane.  In particular, one has that on $C'=\partial\Sigma'$,
$$
R_{12}=R_{34}=0 \mbox{ and } \partial_{\nu'} R_{13}=\partial_{\nu'} R_{14}=\partial_{\nu'} R_{23}=\partial_{\nu'} R_{24}=0.
$$
Let
$$
\mathcal{J}'_-=\set{u\in C^\infty (\Sigma'): L_{\Sigma'} u=0, u\circ \zeta =-u}
$$
be the space of odd elements in the kernel of $L_{\Sigma'}$ -- i.e., odd Jacobi functions.  We denote the space of such elements that vanish on $\partial \Sigma'$ by
$$
\mathcal{D}'_-=\set{u\in \mathcal{J}'_-: u|_{\partial \Sigma'}=0}\subset \mathcal{J}'_-.
$$
Then $R_{12}, R_{34}\in \mathcal{D}_-'$ and the nullity can be computed as
$$
\mathrm{null}_D(\Sigma)=\dim \mathcal{D}'_-.
$$

\subsection{Basic facts}
Let us now record some basic results.  We first suppose that $\Sigma$ is a minimal surface in $\mathbb{S}^3$  bounded by $C$ that may be immersed and either orientable or non-orientable.  The following lower bound on index holds for any non-flat $\Sigma$:
\begin{prop}\label{IndLowBndProp} If $\Sigma$ is not a hemisphere, then
	$$\mathrm{ind}_D(\Sigma)\geq 2.$$
\end{prop}
\begin{proof}
	If $\mathrm{ind}_D(\Sigma)\leq 1$, then there must be a non-trivial linear relation
	$$
	a \mathbf{e}_1^\perp+b\mathbf{e}_2^\perp=0, a^2+b^2 =1.
	$$
	Taking the cone over $\Sigma$ gives a minimal cone, $\mathcal{C}$ in $\mathbb{R}^4$ bounded by the plane $x_3=x_4=0$. This cone, $\mathcal{C}$ contains the line spanned by $a \mathbf{e}_1+b\mathbf{e}_2$, i.e., the cone splits off a line.  One has $\Sigma=\mathcal{C}\cap \mathbb{S}^3$ is smooth only when $\mathcal{C}$ is a half-plane in which case $\Sigma$ is a hemisphere.
\end{proof}
By Proposition \ref{LBIndNullProp}, this lower bound is sharp and is achieved by the Lawson band.
To study other possibly examples of index two it is enough to restrict attention to non-orientable minimal surfaces.
\begin{prop}
	Let $\Gamma$ be a, possibly immersed, oriented minimal surface bounded by $C$.  The following is true:
	\begin{enumerate}
		\item \label{HS} If $\Gamma$ is embedded, then $\Gamma$ is a hemisphere;
		\item \label{Almgren} If $\Gamma$ has (oriented) genus zero, i.e., is topologically a disk, then $\Gamma$ is a hemisphere;
		\item \label{IndOrient} If $\Gamma$ is not a hemisphere, then $\mathrm{ind}_D(\Gamma)\geq 3$.
	\end{enumerate}
\end{prop}
\begin{proof}
	Item \eqref{HS} was shown by Hardt-Simon \cite{Hardt2011}.  For Item \eqref{Almgren} we note that we may Schwarz reflect $\Gamma$ to obtain a minimal immersion of $\mathbb{S}^2$ into $\mathbb{S}^3$.  By \cite{almgrenInteriorRegularityTheorems1966} (cf. \cite{calabiMinimalImmersionsSurfaces1967}) this is an equatorial sphere and so $\Gamma$ is a hemisphere.  Finally, for \eqref{IndOrient} we note that as $\Gamma$ is oriented, there is a well-defined choice of unit normal on $\Gamma$ and so the functions $n_1$ and $n_2$ are well defined eigenfunctions.  Unless $\Gamma$ is a hemisphere they are linearly independent.  Hence, if $\Gamma$ is not a hemisphere, the multiplicity of the $-2$ eigenvalue of $-L_\Gamma$ is at least two and so there must be a eigenvalue $\mu<-2$ of $-L_{\Gamma}$ with corresponding eigenfunction that does not change sign.  That is, when $\Gamma$ is not a hemisphere, $\mathrm{ind}_D(\Gamma)\geq 3$.
\end{proof}

We now focus on non-orientable $\Sigma$ with $\Sigma'$ denoting the orientation double cover.
We have the following separation property for hemispheres bounded by $C$:
\begin{prop}\label{OrientedProp} For $f_\theta=\cos \theta x_3+\sin \theta x_4$, consider the open subset $H_\theta=\set{f_\theta>0}\cap \Sigma$ of $\Sigma$.  The sets $H_\theta$ are non-empty and
when $\Sigma$ is embedded, they are orientable.
\end{prop}
\begin{proof}
	Without loss of generality we may consider $\theta=0$ so $f_\theta=x_3$.  If $H_0=\set{x_3>0}\cap \Sigma$ were empty, then by the strict maximum principle $\Sigma\subset \set{x_3=0}$ which is only possible if $\Sigma$ were a hemisphere which is excluded.  
	
	Now suppose $\Sigma$ is embedded.  Fix a closed curve $\sigma\subset H_0$.  Choose $\epsilon>0$ small so that $\sigma \subset B_{\pi-\epsilon}^{\mathbb{S}}(\mathbf{e}_3)$ and so  $\partial B_{\pi-\epsilon}^{\mathbb{S}}(\mathbf{e}_3)$ meets $\Sigma$ transversally.  Let $\hat{\Sigma}$ be the component of $\Sigma \cap  B_{\pi-\epsilon}^{\mathbb{S}}(\mathbf{e}_3)$ containing $\sigma$ --  this is an embedded submanifold with boundary.  As  $\partial B_{\pi-\epsilon}^{\mathbb{S}}(\mathbf{e}_3)$ is topologically $\mathbb{S}^2$ and $\partial \hat{\Sigma}\subset \partial B_{\pi-\epsilon}^{\mathbb{S}}(\mathbf{e}_3)$ is a 1-cycle, there is a subset $S\subset \partial B_{\pi-\epsilon}^{\mathbb{S}}(\mathbf{e}_3)$ so $\partial S=   \partial\hat{\Sigma}$.  By rounding out the corners of $S\cup \hat{\Sigma}$ one obtains an embedded closed submanifold of $\mathbb{S}^3$ that contains $\sigma$. 
 and which we continue to denote by $\hat{\Sigma}$.	It follows from Alexander's theorem that $\hat{\Sigma}$ is two-sided and so there is a well defined choice of unit normal to $\hat{\Sigma}$.  Hence, there is a consistent choice of normal along $\sigma$ and so $H_0$ is two-sided and orientable.
	
\end{proof}
We obtain the following two-piece-like property and its generalization:
\begin{prop}\label{x3nodalcountprop}
	Let $f_\theta=\cos \theta x_3+\sin \theta x_4$ and suppose $\Sigma$ has Morse index two.
	\begin{enumerate}
		\item When $\Sigma$ is embedded, $f_\theta$ has exactly two nodal domains.  
		\item When $\Sigma$ is immersed,    $f_\theta$ has between two and $g(\Sigma)+2$ nodal domains.  If the number is strictly larger than two, then at least one domain is non-orientable.
	\end{enumerate}
\end{prop}
\begin{proof}
	By Proposition \ref{OrientedProp}, the sets $\set{f_\theta>0}$ and $\set{f_{\theta}<0}=\set{f_{\theta+\pi}>0}$ must both be non-empty so there are at least two nodal domains.  If $\Sigma$ is embedded, then by Proposition \ref{OrientedProp} each component of both sets is orientable.  When $\Sigma$ is immersed, then at most $g(\Sigma)$ components of the union of the two sets may be non-orientable with the rest orientable.
	
	Consider the function $f'=f_\theta\circ \pi$.  This is an even Dirichlet eigenfunction of $\Delta_{\Sigma'}$ with eigenvalue $-2$.  If $U'$ is a nodal domain of $f'$, i.e., a component of $\Sigma'\setminus \set{f'=0}$, then so is $\zeta (U')$ and $U=\pi(U')$ is a nodal domain of $f_\theta$.  In particular, if $U$ is orientable, then $U'\neq \zeta(U')$, while if $U$ is non-orientable, then $U'=\zeta(U')$.
	
	When $\Sigma$ is embedded each nodal domain of $f_{\theta}$ is orientable and so if there are more than two of them, then $f'$ has at least six nodal domains and none are fixed by $\zeta$.  The same is true when $\Sigma$ is immersed and all the nodal domains are orientable.  Otherwise,  there exists at least one non-orientable domain of $f_{\theta}$ and at most $g(\Sigma)$ of them.  In particular,  if the number of nodal domains is more than $g(\Sigma)+2$, then there are at least three such domains and so $f'$ has at least six nodal domains not fixed by $\zeta$.
		In all cases we may choose three of the domains $\Omega_1, \Omega_2, \Omega_3$ of $f'$ so $\Omega_{-i}=\zeta(\Omega_i)$ form a distinct set.  That is, there are six distinct nodal domains $\Omega_{\pm i}$, $1\leq i \leq 3$ of $f'$ which are involuted by $\zeta$.
	
	For $1\leq i \leq 3$, consider the linearly independent odd functions 
	$$f'_i=\left\{\begin{array}{cc} f' & \Omega_i \\ -f' & \Omega_{-i} \\ 0 & \mbox{otherwise} \end{array} \right.
	$$
	We note that as $\Sigma$ is not totally geodesic, $|A_\Sigma|^2$ is non-zero away from isolated points.
	Hence, integrating by parts  implies
	$$
	\int_{{\Sigma}'}  |\nabla_{\Sigma'} f'_i|^2-|A_{\Sigma'}|^2 (f_i')^2 -2 (f_i')^2< \int_{{\Sigma}'}  |\nabla_{\Sigma'} f'_i|^2-2 (f_i')^2 =0
	$$
	and,	for $i\neq j$,
	$$
	\int_{\Sigma'} \nabla_\Sigma' f'_i \nabla_\Sigma' f'_j -|A_{\Sigma'}|^2 f_i' f_j'- 2 f_i' f_j' =0.
	$$

	Thus, by the variational characterization of eigenvalues, there are three odd eigenfunctions of $L_{\Sigma'}$,  $i=1,2,3$  with eigenvalues $\lambda_i<0$.  That is $\dim \mathcal{U}'_-\geq 3$ which contradicts the index hypothesis	
\end{proof}
For the sake of completeness, we record the following full analog of Ros's two-piece property \cite{rosTwopiecePropertyCompact1995a} in our setting.
\begin{thm}
If $\Sigma$ is any embedded non-orientable minimal surface with boundary $C$, then $\Sigma \setminus \set{f_\theta=0}$ has exactly two components.      
\end{thm}
\begin{proof}
By Proposition \ref{OrientedProp}, there are at least two components of $\Sigma \setminus \set{f_\theta=0}$ and all of them are orientable.  As there are no closed minimal surfaces in $\mathbb{S}^3\setminus \set{f_\theta=0}$, all the components have boundary that lie on the sphere $\set{f_\theta=0}$.   Hence, as $\Sigma^+=\set{f_\theta>0}\cap \Sigma$ is orientable we may fix an orientation on it so that the induced orientation on $\partial \Sigma_+$ is consistent with that induced from the sphere $\set{f_\theta=0}$.  As such, $\Sigma_+  $ splits $\mathbb{S}^3 \cap \set{f_\theta>0}$ into two regions $\Omega_\pm$ with $\partial \Omega_\pm \subset \set{f_\theta=0}\cup \Sigma_+$ labeled so $\Omega_+$ induces the fixed orientation $\Sigma_+$ while $\Omega_-$ induces the opposite one.

Pick a component, $\Gamma$, of $\Sigma^+$.  We note that $\Gamma$ is unstable which can be seen by making a choice of unit normal (which is well defined since $\Gamma$ is orientable) and observing that the second variation is negative when one plugs in $f_\theta$.   As $\Gamma$ is oriented, we may think of it as an integral current and so $\gamma=\partial \Gamma$ is supported in $\bar{\Omega}_\pm$. By the direct method there is a least area integral current $\Upsilon$ in $\bar{\Omega}_+$ with $\partial \Upsilon=\gamma$. By Solomon-White's maximum principle \cite{SolWhite} and the fact that there are no closed minimal surfaces in $\set{f_\theta>0}$, either $\mathrm{spt}(\Upsilon)\subset \partial \Omega_+$ or $\spt(\Upsilon)\setminus \spt(\gamma)\subset \Omega_+$,  In the latter case, $\spt(\Upsilon)$ is a smooth stable minimal surface in $\Omega_+$ by standard regularity theory, but this is not possible for the same reason $\Gamma$ is unstable.  For the same reason, $\spt(\Upsilon)$ is disjoint from $\Sigma_+$ and so we must have $\spt(\Upsilon)\subset \set{f_\theta=0}$.  The same is true in $\Omega_-$.  It follows that $\Gamma$ is the only component of $\Sigma^+$ as any of the other components would act as a barrier.  Arguing in the same way for $f_{\theta+\pi}$ gives the result.
\end{proof}
\section{Euler inequality and some applications}\label{NodalSec}
Let us continue to denote by $\Sigma$ a  non-orientable minimal surface bounded by $C$.  When $u$ is a Dirichlet eigenfunction of $\Delta_\Sigma$ or $L_{\Sigma}$,  it is a basic fact that the nodal set $\set{u=0}\subset \Sigma$ gives rise to a planar graph whose vertices correspond to the critical points of $u$.  At interior points the degree of the vertex is twice the order of vanishing while at boundary points it is the order of vanishing plus one.  There may be finite set of closed nodal lines so in order to make them edges we add one ``fake" degree two vertex to each such closed nodal line.  Clearly, this graph is embedded in $\Sigma$ with the nodal domains of $u$ the faces.  

More formally, associated to $u$ we have a graph $(f,e,v)$ where $f$ is the set of faces, i.e., nodal domains, $e$ the set of edges, i.e., smooth parts of the nodal set and $v=v_I\cup v_\partial\cup f_F$ are the vertices which we decompose into interior vertices, $v_I$ which have degree twice the order of vanishing of $u$, $v_{\partial}$ the boundary vertices which have degree one plus the order of vanishing of $u$ and fake vertices, $v_F$ which have degree two.

Our basic tool, used also in \cite{karpukhinMultiplicityBoundsSteklov2014, rossmanLowerBoundsMorse2002, hoffmann-ostenhofMultiplicityEigenvaluesLaplacian1999} is the Euler inequality which states
\begin{equation}\label{EulerIneqSigEqn}
|f|-|e|+|v| \geq \chi(\Sigma)=1-g(\Sigma)
\end{equation}
where $|f|$ is the number of faces, $|e|$ the number of edges and $|v|$ the number of vertices and $\chi(\Sigma)$ the Euler characteristic of $\Sigma$.  Moreover, this inequality is strict unless each face, i.e., each nodal domain is topologically a disk.  By construction,  the vertices in $v_{\partial}$ have degree at least $3$ and the vertices in $v_I$ have degree at least $4$ while the vertices in $v_F$ have degree $2$.  It follows from the degree sum formula that
\begin{equation}\label{EdgeEst}
|e|\geq \frac{3}{2} |v_{\partial}|+ 2 |v_I|+|v_F|.
\end{equation}

Likewise, we note that for a Dirichlet eigenfunction, $u'$, of $\Delta_{\Sigma'}$ or $L_{\Sigma'}$, on  $\Sigma'$  the same construction holds for the nodal set of $u'$ and gives a graph with $(f',e',v')$ satisfying
\begin{equation}\label{EulerIneqSigPrimeEqn}
|f'|-|e'|+|v'|\geq \chi (\Sigma')=2-2g(\Sigma).
\end{equation}
where again the inequality is strict unless each face, i.e., nodal domain of $u'$, is topologically a disk.

\begin{prop} \label{NodalCountWProp} Suppose $\Sigma$ is a non-orientable minimal surface bounded by $C$ of index two. 
	\begin{enumerate}
		\item If $\Sigma$ is embedded, then $|e(\Sigma)|\leq g(\Sigma)+1$;
		\item If $\Sigma$ is immersed, then $|e(\Sigma)|\leq 2g(\Sigma)$.
     \end{enumerate}  
     In particular, an immersed minimal M\"obius band of index two has $|e(\Sigma)|\leq 2$.
\end{prop}
\begin{proof}
	Consider the hemispheres $H_\pm$ given by $x_3=0$ and $\pm x_4\geq 0$.  The number of times the outward conormal $\nu$ points into $H_+$ is bounded from below by $|e(\Sigma)|$ and the same is true of $H_-$.  It follows that there are at least  $2|e(\Sigma)|$ points where $\nabla_\Sigma x_3=0$ on $C$.  Consider the graph associated to the nodal set of $x_3$ so we have $|v_\partial|\geq 2|e(\Sigma)|$.  Hence, it follows from \eqref{EulerIneqSigEqn} and \eqref{EdgeEst} that
$$
|f|-|e(\Sigma)|\geq |f|-|e|+|v|\geq 1-g(\Sigma).
$$

When $\Sigma$ is embedded and index two, Proposition \ref{x3nodalcountprop} implies $|f|=2$ and so
$$
g(\Sigma)+1\geq |e(\Sigma)|.
$$
When $\Sigma$ is an immersed and each face is a disk the same bound holds and we may use that $g(\Sigma)\geq 1$ implies $2g(\Sigma)\geq g(\Sigma)+1$ to obtain the second claim in this case.  Finally, when $\Sigma$ is immersed and the number of faces is more than three, then at least one is not a disk and Proposition \ref{x3nodalcountprop} implies $|f|\leq g(\Sigma)+2$ while the strictness of the Euler inequality implies
$$
2g(\Sigma)\geq |e(\Sigma)|.
$$

\end{proof}

We next establish gap phenomena for the number of nodal domains of odd eigenfunctions associated to  M\"{o}bius bands:
\begin{lem}\label{OddEigenFuncNDLem}
	Suppose $g(\Sigma)=1$ and  $u\in C^\infty({\Sigma'})$ is a non-trivial odd eigenfunction of $L_{\Sigma'}$ with Dirichlet boundary conditions.  The number of nodal domains of $u$ is an even integer.   Moreover, if each nodal domain is topologically a disk, then the number is either two or at least six.  
\end{lem}
\begin{proof}
	Let $Z'=\set{u=0}$ be the nodal set of $u$ which is a finite graph.  As $u\circ \zeta=-u$,  $\zeta(Z')=Z'$ so $Z=\pi (Z')\subset \Sigma$ is a well defined graph as well.  We note that if $U'$ is a component of ${\Sigma'}\setminus Z'$ on which $u>0$ one has $u<0$ on $\zeta(U')$.  This means there are an even number of components of ${\Sigma'}\setminus Z'$, i.e., an even number of nodal domains.  Another property is that if $p\in \partial U' \setminus \partial {\Sigma'}\subset Z'$, then any neighborhood of $p$ meets a nodal domain other than $U'$.
	
	Now suppose all the nodal domains are topological disks. We treat two cases.  The first is that one of the choices of $U'$ has the property that $\pi|_{\bar{U}'}$ is not a homeomorphism onto its image.  As $\pi$ is a closed continuous map, this means there are distinct points $p_\pm\in \bar{U}_+'$ with $\pi(p_+)=\pi(p_-)$.  Connecting $p_-$ to $p_+$ by some path ${\Gamma}$ in $\bar{U}'$ one has that $\gamma=\pi\circ\Gamma$ is a closed curve in $\Sigma$ homotopic to the core of the M\"obius band.  It follows that $\Sigma \setminus \pi(\bar{U}')$ and hence also $\Sigma \setminus(\pi(\bar{U}'\cup Z'))$ consists of planar domains.  Moreover,  this set is either empty, an annulus or has at least two connected components.  The number of components of ${\Sigma'}\setminus ( U'\cup \zeta(U')\cup Z')$ is at least twice the number of components of $\Sigma \setminus \pi(\bar{U}'\cup Z')$ and so is at least four unless $\Sigma \setminus \pi(\bar{U}'\cup Z')$ is empty or an annulus.  However, disks can't cover annuli so either the set is empty or has at least four components.  That is, $u$ has either two nodal domains or at least six.
	
	In the second case we may assume that  $\pi|_{\bar{U}'}$ is a homeomorphism onto its image for each nodal domain $U'$.  Suppose for the sake of contradiction there are four components of ${\Sigma'}\setminus Z'$. This implies there are two components $U_\pm$ of $\Sigma\setminus Z$.  Now let $U_\pm'$ denote the two nodal domains of $u$ on which $u>0$ labeled so that $\pi(U_\pm')=U_\pm$.  By construction, $U_+'$ cannot share an edge with $U_-'$.   Moreover,  $U_+'$ cannot share an edge with $\zeta (U_+')$ as if $p'$ was a point on such an edge, then $\zeta(p')$ would also have this property.  That is $p',\zeta(p')\in \bar{U}_+'$ contradicting the hypothesis that $\pi$ was injective on $\bar{U}_+'$.  Hence, $U_+'$ may share an edge only with $\zeta(U_-')$ and vice versa.  Likewise, $U_-'$ may share an edge only with $\zeta(U_+')$.
	
	Now let $v$ be a vertex of $Z'$, i.e., a point where $u$ has a critical point.  We claim any sufficiently small neighborhood of $v$ meets only $V_+'=U_+'\cup \zeta(U_-')$ or $V_-'=U_-'\cup \zeta (U_+')$ but not both.  Indeed, if this is not true then one can find a point $p'$ near $v$ on an edge of $Z'$ that meets both $V_+'$ and $V_-'$ but this is not possible by the previous paragraph.   It follows that $\bar{V}_+'$ is disjoint from $\bar{V}_-'$, but both are non-empty sets and so this is only possible if $\Sigma'$ is disconnected which is the desired contradiction.

	%
\end{proof}

\begin{prop}\label{JacobiNodalProp}
	If $g(\Sigma)=1$ and $\Sigma$ has Morse index two, then $\mathrm{null}_D(\Sigma)\leq 3$. 
\end{prop}
\begin{proof}
	Let $C_+'$ be one of the components of $\partial {\Sigma'}$ which can be identified with the great circle $C$.
	
	If $\mathrm{null}_D(\Sigma)=\dim \mathcal{D}'_->3$ is strictly greater than three, then, by the Calderon unique continuation theorem, there is a non-zero element $u\in \mathcal{D}'_-$ so that on $C_+'$, 
	$$
	\int_{C_+'} \partial_\nu u ds=\int_{C_+'} \partial_\nu u \cos s ds = \int_{C_+'} \partial_\nu u \sin s  ds=0
	$$
	where here $s$ is the arclength parameter of $C_+'$.
	That is the outward normal derivative of $u$ is orthogonal to the constant function and the two simplest trigonometric polynomials.
	
	It follows by the Sturm-Hurwitz theorem that $\partial_\nu u$ has at least four zeros on $C_+$.  As $\zeta$ maps $C_+'$ to $C_-'$, the other component of $\partial \Sigma$, and $u$ is odd the same is true on $C_-'$.  Hence, $\nabla_{\Sigma'} u$ vanishes on at least eight points on  $\partial {\Sigma'}$.
	
	Let us denote by $v_\partial'$ the critical points of $u$ on $\partial {\Sigma'}$ so $|v_\partial'|\geq 8$.
	It follows from \eqref{EdgeEst} and the Euler inequality \eqref{EulerIneqSigPrimeEqn} that the number of nodal domains, $|f'|$ of $u$ satisfies
	$$
	|f'|-4\geq	|f'|-\frac{1}{2}|v_\partial'|\geq |f'|-|e'|+|v'|\geq \chi({\Sigma'})=0.
	$$
	Moreover, this inequality is strict unless each face, i.e., each component of ${\Sigma'}\setminus Z'$ is a disk.  That is, either
	$
	|f'|>4$
	or $|f'|=4$ and each face is a disk.  However, the latter is ruled out by Lemma \ref{OddEigenFuncNDLem}.  Hence, as $u$ is odd, the number of faces is even and so we conclude that $|f'|\geq 6$.   Arguing as in the proof of Proposition \ref{NodalCountWProp} implies there are at least three odd Dirichlet eigenfunctions of $L_{\Sigma'}$ with negative eigenvalue which contradicts the index assumption.
\end{proof}

\section{Geometry of the boundary and some consequences}

We carry out some geometric computations for $\Sigma$ at $C$ and use it to study the index.    Let
$$
\mathbf{T}=x_2 \mathbf{e}_1-x_1\mathbf{e}_2=x_2 \nabla_\Sigma x_1-x_1\nabla_\Sigma x_2
$$
be a choice of unit tangent to $C$.  In the same vein, we denote the outward conormal by
$$
\nu=\cos \theta(s) \mathbf{e}_3+\sin \theta(s)\mathbf{e}_4.
$$
We then have a choice of unit normal to $\Sigma$ along $C$
$$
\mathbf{n}= -\sin \theta(s) \mathbf{e}_3+\cos\theta(s) \mathbf{e}_4
$$
where we use the convention that the frame
$$
( \mathbf{x}, \nu , \mathbf{T}, \mathbf{n} )
$$
is positively oriented on $C$. We then compute that 
$$
\kappa=\theta'(s)=\mathbf{n}\cdot \nabla_{\mathbf{T}} \nu = -\nu \cdot \nabla_{\mathbf{T}} \mathbf{n}=\mathbf{n}\cdot {A}_{\Sigma}(\mathbf{T}, \nu).
$$
As $C$ is a geodesic in $\mathbb{S}^3$ and $\Sigma$ is minimal
$$
0={A}_{\Sigma}(\mathbf{T}, \mathbf{T})={A}_{\Sigma}(\nu, \nu).
$$

A first, elementary, consequence:
\begin{prop}\label{NoWindZeroProp}
	If $g(\Sigma)=1$, then $e(\Sigma)\neq 0$.
\end{prop}
\begin{rem}
	In general, if $\Sigma$ is embedded, and not necessarily minimal, one has $e(\Sigma)\neq 0$.  However,  an appropriate modification of Boy's surface with a small disk removed shows there are immersed M\"{o}bius bands with winding zero and so minimality is essential.
\end{rem}
\begin{proof}
  Consider the doubled surface of $\Sigma$, obtained by Schwarz reflection across $C$,  which we denote by $\hat{\Sigma}$.  This is a minimal immersion of a Klein bottle into $\mathbb{S}^3$.  Taking the double cover gives $\hat{\Sigma}'$ which is a minimal immersion of a torus.  It is a standard result, see for instant, \cite{almgrenInteriorRegularityTheorems1966}, that the second fundamental form is the real part of a holomorphic quadratic differential. Since, $\hat{\Sigma}'$ is not a sphere we have that $|A_{\Sigma}|$ does not identically vanish.  As $\hat{\Sigma}'$ has genus one, it follows that $|A_\Sigma|$ is nowhere vanishing.  In particular, $\kappa\neq 0$ from which the claim follows. 
\end{proof}

We next compute an expansion of $R_{12}$ and $R_{34}$ at points on $C'$.
\begin{lem}\label{ComputationLem}
	On $C'=\partial \Sigma'$ one has (up to sign)
	$$
	\nabla_{\Sigma'} R_{12}= S_{\Sigma'}(\mathbf{T})=A_{\Sigma'}(\mathbf{T},\nu ) \nu=\kappa\nu \mbox{ and }	\nabla_{\Sigma'} R_{34}= \nu.
	$$
	Likewise,
	$$
	\nabla^2 R_{12}(X,Y)=-(\nabla_{\mathbf{T}} A_{\Sigma'})(X,Y)=(\mathbf{T}\cdot \kappa)((\mathbf{T}\cdot X) (\mathbf{v} \cdot Y))+(\mathbf{T}\cdot Y) (\mathbf{v} \cdot X))
	$$
	while
	$$
	\nabla^2 R_{34} (X,Y)= 0.
	$$
\end{lem}
\begin{proof}
	Write $\partial \Sigma'=C'=C'_+\cup C'_-$.   We compute that, up to relabelling, on $C'_+=\partial \Sigma'$
	$$
	\nabla_\Sigma x_1= x_2 \mathbf{T}, \nabla_\Sigma x_2 =-x_1 \mathbf{T}.
	$$
and
$$
	\nabla_\Sigma x_3= \cos \theta(s) \nu= n_4 \nu, \nabla_{\Sigma} x_4=\sin \theta(s) \nu =-n_3 \nu.
	$$
As such
$$
\nabla_{\Sigma'} R_{12}= (x_2 n_2 +x_1 n_1)\mathbf{T}-S_{\Sigma'}( x_1 \mathbf{e}_2^\top - x_2 \mathbf{e}_1^\top)= S_{\Sigma'}(\mathbf{T})= \kappa \nu.
$$ 
Similarly, 
$$
\nabla_{\Sigma'} R_{34} =-S_{\Sigma'}(x_3 \mathbf{e}_4^\top - x_4 \mathbf{e}_3^\top) +(n_4^2 + n_3^2) \nu= \nu.
$$

We have on $C'$
$$
\nabla^2_{\Sigma'} R_{12}(\mathbf{T}, \mathbf{T})= \mathbf{T}\cdot (\mathbf{T}\cdot R_{12}))- \nabla^{\Sigma'}_{\mathbf{T}}\mathbf{T} \cdot R_{12}=0.
$$
As $L_{\Sigma'} R_{12}=0$, and $R_{12}=0$ on $C'$, it follows that
$$
\nabla^2_{\Sigma'} R_{12}(\nu, \nu)=0.
$$
Using the previous computations we have
$$
\nabla^2_{\Sigma'} R_{12}(\mathbf{T}, \nu)= \mathbf{T}\cdot \partial_\nu R_{12}- \nabla^{\Sigma'}_{\mathbf{T}} \nu \cdot R_{12}=  \mathbf{T}\cdot \kappa
$$
Where we used that $C$ was a geodesic so $\nabla^{\Sigma'}_{\mathbf{T}} \nu=0$.
This gives the formula for $\nabla^2_{\Sigma'} R_{12}$.

For the same reason
$$
\nabla^2_{\Sigma'} R_{34}(\mathbf{T}, \mathbf{T})=\nabla^2_{\Sigma'} R_{34}(\nu, \nu)=0
$$
while
$$
\nabla^2_{\Sigma'} R_{34}(\mathbf{T}, \nu)=\mathbf{T} \cdot \partial_\nu R_{34} -  \nabla^{\Sigma'}_{\mathbf{T}} \nu \cdot R_{34}=0.
$$
This completes the proof.
\end{proof}

We note the following elementary fact:
\begin{prop}\label{NullProp}
If $R_{12}$ and $R_{34}$ are linearly dependent, then up to reparameterizing and a rotation, $\Sigma$ agrees with the parameterizations of $\tau_{1,2k}$ given by Lawson \cite{lawsonCompleteMinimalSurfaces1970}.  In particular, if $\Sigma$ is embedded or of index two, then, up to symmetries $\Sigma=\Sigma_{LB}$.
 
As a consequence,  
$$1\leq\mathrm{null}_D(\Sigma)$$ 
and equality holds when and only when $\Sigma$ is ruled.
\end{prop}
\begin{proof}
By Lemma \ref{ComputationLem}, $R_{34}$ is not identically zero. 
 Moreover, after excluding hemispheres, the proof of Proposition \ref{NoWindZeroProp} implies $A_{\Sigma'}$ vanishes at isolated points.  In particular, by Lemma \ref{ComputationLem} in this case $R_{12}$ is also not identically zero.  It follows that $1\leq\mathrm{null}_D(\Sigma)$.

 Now suppose $R_{12}$ and $R_{34}$ are linearly dependent, i.e., 
 $$
 a R_{12}+b R_{34} \equiv 0
 $$
 where both $a$ and $b$ are non-zero constants.  By Lemma \ref{ComputationLem}, 
 $$
 a \kappa +b\equiv 0.
 $$
 In particular, by integrating we see that $\alpha = \frac{b}{a}$ is a non-zero integer.
 
Consider the minimal immersions 
 $$
 \mathbf{F}_\alpha(x,y)= (\cos y\cos x, \sin y \cos x, \cos \alpha y \sin x, \sin \alpha y \sin x),
 $$
 One verifies that these immersions map $x=0$ to $C$ and along this line one has
 $$
 \kappa_\alpha + \alpha =0.
 $$
 It follows that, up to rotating in the $x_1x_2$ plane,  the surface parameterized by $\mathbf{F}_\alpha$ and $\Sigma$ both parameterize $C$ and have normals that agree there and so by the Calderon unique continuation theorem parameterize the same minimal surface.
 
 As observed by Lawson \cite{lawsonCompleteMinimalSurfaces1970}, $\Sigma_{LB}=\bar{\tau}_{1,2}$ is the only such surface which is embedded.  Likewise, by Propositions \ref{NodalCountWProp} and  \ref{LBIndNullProp}, $\Sigma_{LB}$ is the only one of index two.
Finally, $\mathrm{null}_D(\Sigma)\geq 2$ unless $R_{12}$ and $R_{34}$ are linearly dependent.
\end{proof}

We now show a characterization of the Lawson band in terms of its winding speed.
\begin{prop}\label{NoSlowSpeedCharProp}
Suppose that $g(\Sigma)=1$ and $\Sigma$ has index two.  If $|\kappa|\geq 1$, then $\Sigma$ is, up to symmetries, equal to $\Sigma_{LB}$.
\end{prop}
\begin{rem}
The condition $\kappa\geq 1$ has the following interpretation. Consider $\tau_{1,1}^\theta$, the $\mathbb{S}^1$ family of Clifford tori that contain $C$ and the polar circle $C^*$.  As in \cite{HR}, if $\bar{\tau}_{1,1}^\theta$ is a consistent choice of component of  $\tau_{1,1}^\theta \setminus (C\cup C^*)$, then this collection foliates $\mathbb{S}^3\setminus (C\cup C^*)$.  For $\epsilon>0$ small, let $\sigma=\Sigma \cap \partial T_\epsilon(C)$ which is a $(1,2)$-curve in the torus given as the boundary of the $\epsilon$-tubular neighborhood of $C$.  Roughly speaking, $\kappa$ is the slope of this curve.  Consider the curves $\tau_\theta$ given by  $\bar{\tau}_{1,1}^\theta \cap \partial T_\epsilon(C)$ these are $(1,1)$-curves of slope exactly 1 and form a foliation of the torus.  For topological reasons, $\sigma$ and $\tau_\theta$ intersect at least once. When $\kappa>1$, this is the exact number of intersections, however when $\kappa<1$ there are at least some $\theta$ so $\sigma \cap \tau_\theta$ has at least three intersections. 
\end{rem}
\begin{proof}
  Without loss of generality we may assume $\kappa=\theta'(s)\geq 1$ as otherwise $\kappa\leq -1$ and we may reflect to switch to the first case.	 We note that by Proposition \ref{NodalCountWProp},  as $g(\Sigma)=1$ and has index two, $e(\Sigma)=2$.  That is,
  $$
  \int_0^{2\pi} \kappa =4\pi.
  $$
	
  On $\Sigma$ recall that $x_1$ and $x_2$ are Neumann eigenfunctions of $\Delta_\Sigma$ with eigenvalue $2$ while $x_3$ and $x_4$ are Dirichlet eigenfunctions with the same eigenvalue.  It follows from integrating by parts that
  $$
  0=\int_{\Sigma} x_1 \Delta_\Sigma x_3 - x_3 \Delta_\Sigma x_1= \int_{\partial \Sigma} x_1 \partial_\nu x_3 - x_3 \partial_\nu x_1 =\int_{\partial \Sigma} x_1 \partial_\nu x_3
  $$
 Repeating this, we obtain
 $$
 0= \int_{\partial \Sigma} x_2 \partial_\nu x_3= \int_{\partial \Sigma} x_1 \partial_ \nu x_4= \int_{\partial \Sigma} x_2 \partial_\nu x_4.
 $$
 Writing $x_1=\cos s$ and $x_2 =\sin s$ this means
 $$
0= \int_0^{2\pi} \cos s \cos \theta(s) =\int_0^{2\pi} \sin s \cos \theta(s) =\int_0^{2\pi} \cos s \sin \theta(s)=\int_0^{2\pi} \sin s \sin \theta(s).  
 $$
Rearranging terms implies
$$
0=\int_0^{2\pi} \cos (\theta(s)-s)= \int_0^{2\pi}\sin (\theta(s)-s)
$$
Let us now suppose first that $\theta'(s)>1$.  As $\int_0^{2\pi} \theta'(s) =4\pi$ it follows that
$$
u=\psi(s)= \theta(s)-s
$$
may be thought of as a diffeomorphism of $\mathbb{S}^1$ to itself.   It follows from the change of variables formula that
$$
0=\int_0^{2\pi} (\psi^{-1})'(u)\cos u =\int_0^{2\pi} (\psi^{-1})'(u)\sin u .
$$
As $\psi^{-1}$ is a also a diffeomorphism of $\mathbb{S}^1$ to itself we have 
$$
0=\int_0^{2\pi} \left( (\psi^{-1})'(u)-1 \right).
$$

It follows from the Sturm-Hurwitz theorem that
$$
(\psi^{-1})'(u)-1=\frac{1}{\psi'(\psi^{-1}(u))}-1
$$
has four zeros on $[0, 2\pi)$.   This implies $ \psi'(s)-1= \kappa(s)-2$ has four zeros.

Hence, by Lemma \ref{ComputationLem}, 
$$
\partial_\nu (R_{12} -2  R_{34})
$$
has at least eight zeros on $C'=\partial \Sigma'$ -- four on each of $C'_+$ and $C'_-$.

We now consider the nodal graph associated to $u= R_{12}-2R_{34}$.  Either $u$ identically vanishes in which case we are done by Proposition \ref{NullProp} or this nodal graph has at least eight boundary vertices.  It follows from the Euler inequality that this would mean 
$$
|f'|-4\geq |f'|-\frac{1}{2} |v_{\partial}'| \geq 0.
$$
Arguing as in the proof of Proposition \ref{JacobiNodalProp} shows this is inconsistent with being index two.

It remains to treat the case $\theta'(s)\geq 1$.   We claim it is possible to approximate $\psi$ by  diffeomorphisms $\psi_\epsilon$ of $\mathbb{S}^1$ that satisfy
$$
0=\int_0^{2\pi} \cos \psi_{\epsilon}(s) =  \int_0^{2\pi} \sin \psi_{\epsilon}(s)=\int_0^{2\pi} \left(\psi_\epsilon'(s)-1\right)
$$
it follows that $\psi_\epsilon'-1$ has at least four zeros and so, counting mulitplicity, this continues to hold for $\psi$.  The rest of the proof goes through with minor modifications.  One way to see the existence of such an approximation is to note that the hypotheses on $\psi$ ensure it corresponds to a closed curve in $\Real^2$ with geodesic curvature $\kappa-1 \geq 0$ and length $2\pi$. This correspondence is via integrating the unit tangent $\mathbf{T}=(\cos \psi, \sin \psi)$.  One may then run length normalized curve shortening flow to obtain approximating strictly convex curves -- these determine the $\psi_\epsilon$ by undoing the correspondence.
\end{proof} 

\section{Index two minimal surfaces and symmetry}\label{SymSec}
We consider restrictions on index two minimal surfaces bounded by $C$ that posses some degree of discrete symmetry.
We first establish the existence of a certain element in $\mathcal{D}'_-$:
\begin{lem}\label{VanishingOrderLem}
	If $\kappa$ is not constant on $\partial \Sigma'$, then there is choice of $\theta$ so that
	$$
	R_\theta= \cos \theta R_{12} +\sin \theta R_{34}
	$$
	has a zero of order at least three at two distinct points of  $p_{\pm}\in C_{\pm}'\subset \partial \Sigma'$.
\end{lem}
\begin{proof}
	Consider the function $\kappa'= A_{\Sigma'}(\mathbf{T}, \nu)$ on $C'=\partial \Sigma'$.  If $\kappa'$ is not constant, then there is a point, $p\in C'$,  say at the maximum,  where $\mathbf{T}\cdot \kappa'=0$.
	It follows from Lemma \ref{ComputationLem} that, for all $\theta$,
	$$
	\nabla^2_{\Sigma'} R_\theta|_{p}=0.
	$$
	Now we have
	$$
	\nabla_{\Sigma'}  R_\theta (p)= (-\cos \theta \kappa'(p)-\sin \theta ) \nu(p)
	$$ 
	and so we may find a value of $\theta$ so that
	$$
	\nabla_{\Sigma'} R_\theta(p)=0
	$$
	That is, $R_\theta$ vanishes to order at least three at $p\in \partial \Sigma'$.  As $R_\theta$ is an odd function, this also holds at $\zeta(p)$.  Up to applying $\zeta$ we may assume $p=p_+\in C_+'$ and $p_-=\zeta (p)\in C_-'$.
\end{proof}

If $\Sigma$ contains many meridians relative to its genus then it cannot be index two. 
\begin{thm}\label{SymGenIndThm}
	For $N\geq 1$, suppose there are $2N$ equally spaced points on $C$ so that $\Sigma$ contains $N$ meridians with endpoints on this collection of points.
	If $\Sigma$ has index two, then either $\Sigma=\Sigma_{LB}$, $g(\Sigma)=1$ and $N=1$ or $g(\Sigma)\geq  \min\set{2,N-1}$.
\end{thm}
\begin{proof}
	If $\kappa$ is constant, then $R_{12}$ and $R_{34}$ are linearly dependent by the Calderon unique continuation theorem and so, by Proposition \ref{NullProp}, $\Sigma=\Sigma_{LB}$. 	In what follows we suppose $\kappa$ is not constant.
	
	Let us denote the set of meridian by $M_i$, $1\leq i \leq N$ and let $\mathcal{R}_{M_i}$ be the symmetry given by reflection around $M_i$.  By the Schwarz reflection principle, these are symmetries of $\Sigma$.  They generate the dihedral group ${D}_{N}$.  
	
	We note that $\kappa$ is preserved by the symmetries of $D_{N}$.  We treat two cases:
	\begin{enumerate}
		\item $\kappa$ has critical points only on the set of endpoints of the $M_i$;
		\item $\kappa$ has a critical point off of the set of endpoints of the $M_i$.
	\end{enumerate}	
	In the first case, as $\kappa$ only has a critical point on an an endpoint, then, passing to the oriented double cover, the symmetry implies that one may choose $\theta$ so $R_\theta$ has $2N$ zeros of order $\ell\geq 3$ at these points.  These are the endpoints of the lifts of the meridians -- specifically, when $N$ is even they lie on both endpoints of alternating meridians, while when $N$ is odd they lie on alternating endpoints of each meridian.  In any case, the Euler inequality implies
$$
|f|-2N \geq 2-2g(\Sigma)
$$
and the inequality is strict unless all the domains are disks.   When $g(\Sigma)=1$, the index hypothesis and Lemma \ref{OddEigenFuncNDLem} imply either $|f|= 2$ or $|f|\leq 4$ and the inequality is strict.  That is,  
$2=2g(\Sigma)\geq 2N$ or $2+2g(\Sigma)=4>2N$.  In either case, one has $N=1$.	
For $g(\Sigma)\geq 2$, the index hypothesis implies $|f|\leq 4$ so 
$$
2g(\Sigma)\geq 2N-2
$$
which gives the claimed bound.
	
	In the second case, on the oriented double cover we see that $\kappa$ has at least $2N$ critical points on $C'_+$.  By Lemma \ref{VanishingOrderLem} we may pick $\theta$ so that $R_\theta$ vanishes to order at least three at these points.  That is there are $4N$ points on $\partial \Sigma'$ where $R_\theta$ has a zero of order at least three.
	By the Euler inequality applied to the graph associated to the nodal set of $R_{\theta}$ one obtains
	$$
	|f| -4N\geq 2- 2g(\Sigma)
	$$
	So either $|f|\geq 2+4N -2g(\Sigma)$ and all faces are disks or $|f|>2+4N-2g(\Sigma)$.  Suppose $g(\Sigma)=1$, the index hypothesis and Lemma \ref{OddEigenFuncNDLem} imply either $|f|= 2$ or $|f|\leq 4$ and the inequality is strict.  That is, either $2=2g(\Sigma)\geq 4N$ or $2+2g(\Sigma)=4>4N$.  As neither of these is possible, we must have $g(\Sigma)\geq 2$.  In this case, the index hypothesis implies $|f|\leq 4$
	from which it follows that $g(\Sigma) \geq  2N-1>N-1$.

%
%
\end{proof}
We may now prove Theorems \ref{SymRigidThm} and Theorem \ref{BKIndThm}.
\begin{proof} [Proof of Theorem \ref{BKIndThm}]
Consider the subgroup of $SO(4)$ generated by reflections across two of the distinct meridians.  If this subgroup is finite, then it contains a subgroup isomorphic to the dihedral group $D_2$ and the result is an immediate consequence the $N\geq 2$ case of Theorem \ref{SymGenIndThm}.   If the subgroup is infinite, then $\Sigma$ contains an infinite number of meridians.  We observe that $\mathbf{T}\cdot \kappa'$ is zero at each end point of a meridian.  As there must be a point $p\in \partial \Sigma$ that is an accumulation point of the end points of the meridians.  It follows that $\mathbf{T}\cdot \kappa'$ vanishes to infinite order at $p$ and so, as $\kappa'$ is real analytic, $\kappa'$ is constant and the result follows.
\end{proof}
\begin{proof}[Proof of Theorem \ref{BKIndThm}]
	For $k\geq 1$, $m\geq 3$ odd, the surfaces $\tilde{\xi}_{m,k}$ are embedded non-orientable and bounded by $C$.  They contain $m(k+1)$ meridians and have genus $mk+1$.  As $m\geq 3$,
	$$
	4\leq g(\tilde{\xi}_{m,k})=mk+1< mk+m-1=m(k+1)-1
	$$
	and so the surface cannot have index two by Theorem \ref{SymGenIndThm}.
\end{proof}
\section{Minimal punctured Klein bottles of index two}
There are no known examples of an embedded minimal surface bounded by $C$ with $g(\Sigma)=2$ -- the examples from \cite{BKgenus} start with $g(\Sigma)=4$.  In this section we rule out the existence of any such surface with index two.
We first recall a basic consequences complex analysis has on the norm of the second fundamental form of $\Sigma$:
\begin{lem}\label{AzerosLem}
 Let $\Sigma$ a minimally immersed surface $C$ with $g(\Sigma)=2$, then one of the following holds (counting multiplicity):
 \begin{enumerate}
 	\item $|A_\Sigma|=0$ at no points on $C=\partial \Sigma$ and two points in the interior;
 	\item $|A_\Sigma|=0$ at two points on $C$ and one point in the interior;
 	\item $|A_\Sigma|=0$ at four points of $C$ and no points in the interior
 \end{enumerate}
\end{lem}
\begin{proof}
	Schwarz reflect $\Sigma$ across $C$ and take the double cover of the result immersion, this gives a minimal immersion of a genus three oriented surface into $\mathbb{S}^3$.  The second fundamental form is the real part of a holomorphic quadratic differential and so, by Riemann-Roch, has eight zeros.  The three cases correspond to how many of the zeros lie on $C$.
\end{proof}

\begin{lem}\label{NonZeroLem}
	If $g(\Sigma)=2$ and the index is two,  then $|A_{\Sigma}| \neq 0$ on $C=\partial \Sigma$.
\end{lem}
\begin{proof}
 On $\Sigma'$ consider  $N_i=\mathbf{n}\cdot \mathbf{e}_i$, $i=1,2$ which are two odd eigenfunctions of $-L_{\Sigma'}$ with eigenvalue $-2$.
 One computes that on $\partial \Sigma'$,  $\nabla_{\Sigma'} N_1=-S_{\Sigma'}(\nabla_{\Sigma'} x_1)=-x_2 S_{\Sigma'}(\mathbf{T})$ and $\nabla_{\Sigma'} N_2 =-S_{\Sigma'}(\nabla_{\Sigma'} x_2)=x_1 S_{\Sigma'}(\mathbf{T})$.  
 For fixed $\theta$, let $N_\theta= \cos \theta N_1 +\sin \theta N_2$.  The previous computations imply that the number of zeros on $C'$, counting multiplicity, of $\nabla_{\Sigma'} N_\theta$ is four more than the number of zeros of $|A_{\Sigma'}|$ on $C'$.

We now suppose $|A_\Sigma|=0$ somewhere on $C$ in order to obtain a contradiction. Combining Lemma \ref{AzerosLem} with the above we see that either $\nabla_{\Sigma'} N_\theta$ has eight zeros on $C'$ and two in the interior or twelve zeros on $C'$.
Let us treat the second case first:  In this case $A_{\Sigma'}$ has eight zeros on $C'$ and so if we pick an arbitrary $\theta$ and see that the graph corresponding the nodal set of $N_\theta$ satisfies $|v_{\partial}'|\geq 12$.  Hence, by the Euler inequality
$$
|f'|-6\geq |f'|- \frac{1}{2}|v_{\partial}'|\geq |f'|-|e'|+|v'|\geq \chi(\Sigma')=2-2g(\Sigma)=-2.
$$ 
It follows that $|f'|\geq 4$ which, by an argument along the lines of the Courant nodal domain theorem implies there was an odd eigenfunction of the stability operator on $\Sigma'$ with eigenvalue below $-2$.  That is, the index of $\Sigma$ would be at least three which a contradiction.

 If instead $|A_\Sigma|$ has only two zeros on $C$, then there are interior points $p_\pm$ where $A_{\Sigma'}$ has zeros and so $p_+=\zeta (p_-)$. 
By choosing $\theta$ appropriately, we can ensure $N_\theta(p_\pm)= 0$.  As
$$
\nabla_{\Sigma'} N_\theta= -S_{\Sigma'}(\cos \theta \nabla_{\Sigma'} x_1+ \sin \theta \nabla_{\Sigma'}x_2)
$$
it follows that $N_\theta$ has interior critical points at $p_\pm$.  As a consequence, the nodal graph associated to this $N_\theta$ satisfies $|v_{\partial}'|\geq 8$ and $|v_{I}'|\geq 2$.   The Euler inequality then gives
 $$
 |f'|- 6 \geq |f'|-\frac{1}{2}|v_{\partial}'|-|v_{I}'| \geq \chi(\Sigma')=-2
 $$
 This again leads a contradiction with the index hypothesis and so we conclude there are no zeros of $|A_\Sigma|$ on $C$.
\end{proof}

\begin{thm} \label{NoKBInd2Thm}
	There is no embedded $\Sigma$ with  $g(\Sigma)=2$ and index two.
\end{thm}
\begin{proof}
	Suppose $\Sigma$ was such a minimal surface. By Lemma \ref{NonZeroLem},  $|A_\Sigma|$ has no zeros on $C$.  Hence, the winding speed $\kappa=A(\nu, \mathbf{T})$ is either positive or negative.   It follows that $e(\Sigma)\neq 0$.   As $\Sigma$ is embedded and $g(\Sigma)=2$, it follows that $|e(\Sigma)|\geq 4$.  However, this is not consistent with Proposition \ref{NodalCountWProp} and so no such surface exists.
\end{proof}
\appendix

\section{Integrable systems and partial rigidity of the Lawson band}\label{IntSec}
\subsection{Main Results}
In this section we use some ideas drawn from the integrable systems literature, in particular an adaptation of ideas of \cite{pinkallClassificationConstantMean1989} to study Conjecture \ref{NoFakeLBConj}.  In particular we establish:

\begin{thm}\label{FakeLBThm}
	If $g(\Sigma)=1$ and $\Sigma$ has index two and is not $\Sigma_{LB}$, then $\Sigma$  corresponds to a suitable non-constant solution of either \eqref{SimpleODE} or \eqref{MainODE}.
\end{thm}
\begin{proof}
	By Proposition \ref{JacobiNodalProp}, $\mathrm{null}_D(\Sigma)\leq 3$.  While by  Proposition \ref{NullProp}, if $\Sigma$ is not the Lawson band, then $\mathrm{null}_D(\Sigma)\geq 2$.  The claim follows from Proposition \ref{InitODEProp}.
\end{proof}
\begin{rem}
In particular, any potential counterexample to Conjecture \ref{NoFakeLBConj} lies within the finite dimensional space of solutions and parameters of \eqref{SimpleODE} and \eqref{MainODE}.
\end{rem}
\subsection{Integrable Systems}
\label{IntSysSec}
We now wish to study minimal M\"{o}bius bands bounded by a circle in $\mathbb{S}^3$ via integrable systems techniques along the lines of those developed by Pinkall-Sterling \cite{pinkallClassificationConstantMean1989}.
To that end, suppose that $\mathbb{M}$ is an abstract compact M\"{o}bius band and that
$$
\iota:\mathbb{M}\to \Sigma \subset \mathbb{S}^3
$$ 
is a minimal immersion with $\iota:\partial \mathbb{M}\to C$ an embedding.  It follows that there is a  conformal minimal immersion
$$
\mathbf{F}: [0, L_1]\times \Real \to \mathbb{S}^3, L_1>0
$$
which is invariant under the transformations
$$
\zeta(x,y)= (L_1 -x, y+L_2), L_2>0.
$$
Setting
$$
\tau_2(x,y)=(x,y+2L_2)=\zeta(\zeta(x,y))
$$
gives a translation under which $\mathbf{F}$ is also invariant.

The lines $x=0$ and $x=L_1$ map to the boundary circle $C$.  
As $C$ is a geodesics in $\mathbb{S}^3$, then Schwarz reflection allows us to extend $\mathbf{F}$ to a conformal minimal immersion
$$
\mathbf{F}:[-L_1,L_1]\times \Real\to \mathbb{S}^3.
$$ 
More precisely, if we let
$$
\mathcal{R}_C(x_1, x_2, x_3, x_4)=(x_1, x_2, -x_3, -x_4) \mbox{ and } \rho_0(x,y)=(-x,y)
$$
be reflection across $C$ in $\mathbb{S}^3$ and reflection across $x=0$ in $\Real^2$, 
then we may extend using $\mathbf{F}\circ \rho_0=\mathcal{R}_C\circ \mathbf{F}$.
We may also Schwarz reflect across $x=L_1$ via the reflection 
$$
\rho_{1}(x,y)= (2L_1-x,y).
$$
Iterating this process yields a conformal minimal immersion
$$
\mathbf{F}: \Real^2 \to \mathbb{S}^3.
$$
This is invariant under $\zeta, \tau_2, \rho_0, \rho_1$ and the translation $$\tau_1(x,y)=\rho_{1}(\rho_0(x,y))=(x+2L_1,y).$$
Let us denote by 
$$\mathrm{Sym}_M\subset \mathrm{Aff}(\mathbb{R}^2)$$ the subgroup of the group of affine transformations of $\Real^2$ generated by $\zeta, \rho_0, \rho_1$.  We let
$$
\mathrm{Sym}_M^T\subset \mathrm{Sym}_M
$$
be the subgroup generated by $\tau_1$ and $\tau_2$.  Note that
$$
\mathbf{F}\circ \tau=\mathbf{F}, \tau\in \mathrm{Sym}_M^T\mbox{ and } \mathbf{F}\circ \zeta =\mathbf{F}
$$
while if we let $\rho_{n} \in \mathrm{Sym}_M^T$ denote reflection across the line $x=nL_1$ then
$$
\mathbf{F}\circ \rho_{n}=\mathcal{R}_C\circ \mathbf{F}.
$$
If we let $\mathbf{n}_{\mathbf{F}}:\Real^2\to \Real^4$ be a choice of unit normal in $\mathbb{S}^3$ to $\mathbf{F}$, then 
$$
\mathbf{n}_{\mathbf{F}}\circ \tau=\mathbf{n}_{\mathbf{F}}, \tau\in \mathrm{Sym}_M^T \mbox{ and } \mathbf{n}_{\mathbf{F}}\circ \zeta =-\mathbf{n}_{\mathbf{F}}, \mathbf{n}_{\mathbf{F}}\circ \rho_{n}=-\mathcal{R}_C\circ \mathbf{n}_{\mathbf{F}}.
$$

If we restrict to the rectangle $[0, L_1]\times [0, 2 L_2]$, then $\mathbf{F}$ double covers a conformal immersion of the M\"{o}bius band and $\zeta$ is the corresponding deck transformation.  Restricting to the rectangle $[-L_1, L_1]\times [0, 2 L_2]$ this is a minimal immersion of a torus -- in particular fits into the framework of \cite{pinkallClassificationConstantMean1989}.
\subsection{Geometric Computations}
Let us write the conformal metric associated to $\mathbf{F}$ as
\begin{equation} \label{ConfPSEqn}
	g_{\mathbf{F}}=\mathbf{F}^* g_{\mathbb{S}}=\frac{1}{4}e^{2\omega} (dx^2+dy^2).
\end{equation}
We remark this normalization is consistent with choices in the literature (e.g, \cite{hauswirthMeanconvexAlexandrovEmbedded2016}) for minimal surfaces in $\mathbb{S}^3$ but differs slightly from the conventions in \cite{pinkallClassificationConstantMean1989}.

We record some basic facts and properties.  
First of all, the symmetry properties of $\mathbf{F}$ ensure that for $\Psi\in \mathrm{Sym}_M$, $\omega \circ \Psi=\omega$.  In particular,  the symmetry under the reflections $\rho_n\in \mathrm{Sym}_M$ implies 
$$
\partial_x \omega=0
$$
on $x=n L_1$, $n\in \mathbb{Z}$. 
Next, the Laplace operator of $g_{\mathbf{F}}$ satisfies:
$$
\Delta_{\mathbf{F}}=4 e ^{-2\omega} \Delta_0=4e ^{-2u}  (\partial_x^2 +\partial_y^2).
$$
Denote by $A_{\mathbf{F}}$ be the second fundamental form of $\mathbf{F}$ associated to $\mathbf{n}_{\mathbf{F}}$
$$
A_{\mathbf{F}} = \mathbf{n}_{\mathbf{F}}\cdot \partial_{xx} \mathbf{F} dx^2 +\mathbf{n}\cdot \partial_{xy} \mathbf{F} dx dy +\mathbf{n}\cdot \partial_{yy} \mathbf{F} dy^2.
$$
One readily checks that, along $x=0$, $\mathbf{F}$ parametrizes a geodesic in $\mathbb{S}^3$ and so
$$
A_{\mathbf{F}}(\partial_x, \partial_x)=0
$$
Hence, using the symmetry, as $A$ is the real part of a holomorphic quadratic differential on a torus, it is of the form
$$
A_{\mathbf{F}}= C dx dy
$$
for some $C$ -- this $C$ must be non-zero as otherwise the surface would be umbilic and $\mathbf{F}$ would parameterize a subset of a great sphere.  As we are free to scale the domain $\mathbb{R}^2$ and reflect along a coordinate axis, we may suppose $C=\frac{1}{4}$.
That is,
\begin{equation}\label{2ndFFPSEqn}
	A_{\mathbf{F}}= \frac{1}{4}( dx\otimes dy +dy\otimes dx).
\end{equation}
This means
$$
|A_{\mathbf{F}}|^2= 2e^{-4\omega}
$$
and so, by the Gauss-Codazzi equations, the Gauss curvature of $g_{\mathbf{F}}$ satisfies
$$
K=1-\frac{1}{2} |A_{\mathbf{F}}|^2= 1 - e^{-4\omega}.
$$
The conformal change of metric also implies that the Gauss curvature satisfies
$$
K=-4 e^{-2\omega} \Delta_0 \omega
$$
Hence, we obtain the following 
\begin{equation}\label{SinhGordonEqn}
	\Delta_0 \omega +\frac{1}{4} (e^{2\omega}-e^{-2\omega})=\Delta_0 \omega +\sinh \omega \cosh \omega =\Delta_0 \omega+\frac{1}{2} \sinh(2\omega)=0
\end{equation}
which is the sinh-Gordon equation.  
The corresponding equation for Jacobi functions is
\begin{equation}\label{LinSinhGordonEqn}
	\Delta_0 v + v\cosh 2\omega  =0
\end{equation}
with corresponding Dirichlet boundary conditions $v=0$ on $x=nL_1$, $n\in \mathbb{Z}$.
Indeed,
$$
\Delta_{\mathbf{F}}+|A_{\mathbf{F}}|^2+2=4e^{-2\omega} (\Delta_0+\cosh 2\omega).
$$

We note that the outward conormal on the boundary component of $\hat{\Sigma}$ covered by $x=2nL_1$, i.e., $\hat{C}_-$ is given by
$$
\nu= 2 e^{-\omega} \partial_x\mathbf{F}
$$
while the conormal on the boundary component covered by $x=(2n+1)L_1$, i.e., $\hat{C}_+$ is
$$
\nu=- 2 e^{-\omega} \partial_x\mathbf{F}.
$$
The unit tangent on the cover of either boundary component is 
$$
\mathbf{T}=2e^{-\omega} \partial_y\mathbf{F}
$$
which is invariant under $\zeta$ and so descends to $\Sigma$.


\subsection{Pinkall-Sterling Jacobi functions}
Let us denote by $\mathcal{J}^{\mathbf{F}}$ the space of $\mathrm{Sym}_M^T$ invariant Jacobi functions associated to the minimal immersion $\mathbf{F}:\Real^2\to \mathbb{S}^3$ as in Section \ref{IntSysSec}.  Namely,
$$
\mathcal{J}^{\mathbf{F}}=\set{u\in C^\infty(\Real^2): u \circ \tau =u, \tau\in \mathrm{Sym}_M^T, \mbox{ u satisfies \eqref{LinSinhGordonEqn}}}.
$$
By standard elliptic theory, this is a finite dimensional vector space.  The space of odd  Dirichlet Jacobi functions is given by
$$
\mathcal{D}_-^{\mathbf{F}}=\set{u\in \mathcal{J}: u\circ \zeta= -u, u\circ \rho_n =-u}\subset \mathcal{J}^{\mathbf{F}}.
$$
The condition, $u\circ \rho_n=-u$ implies $u=0$ on $x=nL_1$ which is why we refer to these as Dirichlet Jacobi functions.  One readily verifies that $\mathcal{D}_-^{\mathbf{F}}$ is isomorphic to $\mathcal{D}'_-$, the space from Section \ref{NodalSec}.
Likewise, the even Neumann Jacobi functions are given as
$$
\mathcal{N}_+^{\mathbf{F}}=\set{u\in \mathcal{J}: u\circ \zeta= u, u\circ \rho_i =u}\subset \mathcal{J}^{\mathbf{F}}.
$$
The condition, $u\circ \rho_n=u$ implies $\partial_x u=0$ on $x=nL_1$ which is why we refer to these as Neumann Jacobi functions.  

The geometry of $\mathbf{F}$ leads to certain natural elements in $\mathcal{D}_-^{\mathbf{F}}$ and $\mathcal{N}_+^{\mathbf{F}}$.  First of all, if $R_{12}$ denote the Jacobi function associated to rotation in the $x_1x_2$ plane and $R_{34}$ the one associated to rotation in the $x_3x_4$ plane, then
$
R_{12}, R_{34}\in \mathcal{D}_-^{\mathbf{F}}.
$
Furthermore, as observed in \cite{pinkallClassificationConstantMean1989} there are also a hierarchy of Jacobi functions computed in terms of the conformal factor $\omega$.  For instance,
$$
v_1=\frac{1}{2}\partial_x \omega=\frac{1}{2}\omega_x=\mathrm{Re}( \omega_z),
$$
$$
v_2 =\frac{1}{8} \omega_{xxx}-\frac{3}{8} \omega_{xyy}-\frac{1}{4} \omega_x^3 +\frac{3}{4} \omega_y^2 \omega_x=\mathrm{Re}(\omega_{zzz} -\omega_z^2);
$$
$$
w_1=\frac{1}{2}\omega_y=-\mathrm{Im}(\omega_z)
$$
and
$$
w_2=\frac{1}{8} \omega_{yyy} -\frac{3}{8} \omega_{xxy} -\frac{1}{4}\omega_y^3+ \frac{3}{4} \omega_x^2 \omega_y=-\mathrm{Im} (\mathrm{Re}(\omega_{zzz} -\omega_z^2).
$$
In \cite{pinkallClassificationConstantMean1989} this hierarchy was introduced for CMC surfaces in $\Real^3$ -- but holds in the current context as well with minor modifications. The geometry and symmetry of $\mathbf{F}$
implies $$
v_1, v_2\in \mathcal{D}_-^{\mathbf{F}} \mbox{ and } w_1, w_2\in \mathcal{N}_+^{\mathbf{F}}.
$$
Using \eqref{SinhGordonEqn},
\begin{align*}
	v_2 &=\frac{1}{2} \omega_{xxx} +\frac{3}{16} (\sinh(2\omega))_x -\frac{1}{4} \omega_x^3 +\frac{3}{4} \omega_y^2 \omega_x\\
	&=\frac{1}{2} \omega_{xxx} +\frac{3}{8} (\cosh(2\omega)) \omega_x -\frac{1}{4} \omega_x^3 +\frac{3}{4} \omega_y^2 \omega_x.\\
	v_2&=- \frac{1}{2} \omega_{xyy} -\frac{1}{8} \cosh(2\omega) \omega_x -\frac{1}{4} \omega_x^3 +\frac{3}{4} \omega_y^2 \omega_x
\end{align*}
and
\begin{align*}
	w_2&=\frac{1}{8} \omega_{yyy} -\frac{3}{8} \omega_{xxy} -\frac{1}{4}\omega_y^3+ \frac{3}{4} \omega_x^2 \omega_y\\
	&=\frac{1}{2} \omega_{yyy} + \frac{3}{8} \cosh(2\omega) \omega_y -\frac{1}{4} \omega_y^3+ \frac{3}{4} \omega_x^2 \omega_y.
\end{align*}

Let us now consider the subspace
$$
\mathcal{V}=\mathrm{span}(v_1, v_2,R_{12}, R_{34})\subset \mathcal{D}_-^{\mathbf{F}}
$$
which has dimension at most four.
By Proposition \ref{JacobiNodalProp} and the identification with $\mathcal{D}'_-$, if the minimal M\"obius band, $\Sigma$, has index two, then this space has dimension at most three, i.e., there are non-trivial linear relations among the four functions.

To study these relations it is helpful to look at the conormal derivative along the boundary, equivalently the $\partial_x$ derivative along $x=0$.
\begin{lem}\label{ConormLem}
	
	We compute that along $x=0$,
	\begin{align*} 
		\partial_x R_{12}&=-\frac{1}{2}e^{-\omega};\\
		\partial_x R_{34}&=\frac{1}{2}e^\omega;\\
		\partial_x v_1 &= -  \frac{1}{2}\omega_{yy}- \frac{1}{4} \sinh(2 \omega);\\
		\partial_x v_2 & = \frac{1}{2}  \omega_{yyyy} +\frac{5}{8}\cosh(2 \omega)  \omega_{yy} +\frac{5}{8}\sinh(2 \omega)  \omega_y^2-\frac{3}{4}  \omega_y^2 \omega_{yy}+\frac{1}{32} \sinh(4 \omega).
	\end{align*}
\end{lem}
\begin{proof}
	For the first claim we observe that the reflection symmetry of $\mathbf{F}$ across $x=0$ implies on $x=0$ that
	$$
	\partial_x x_1=\partial_x x_2=0
	$$
	and
	$$
	\partial_y x_1=-2 e^{-\omega} x_2, \partial_y x_2 =2 e^{-\omega} x_1
	$$
	It follows from the normalizations of Section \ref{IntSysSec} that
	$$
	\partial_x \mathbf{e}_1\cdot \mathbf{n}_{\mathbf{F}}= -A_{\mathbf{F}}(\partial_x, \mathbf{e}_1^\top)= -\frac{1}{4} \mathbf{e}_1\cdot \partial_y \mathbf{F}=\frac{1}{2} e^{-\omega} x_2
	$$
	$$
	\partial_x \mathbf{e}_2\cdot\mathbf{n}_{\mathbf{F}}= -A_{\mathbf{F}}(\partial_x, \mathbf{e}_2^\top)= -\frac{1}{4} \mathbf{e}_2\cdot \partial_y \mathbf{F}=-\frac{1}{2} e^{-\omega} x_1
	$$
	This means, as claimed
	$$
	\partial_x R_{12} = -\frac{1}{2} e^{-\omega}.
	$$
	The computations for $R_{34}$ are similar.
	
	The third claim follows from \eqref{SinhGordonEqn}, i.e., 
	\begin{align*}
		\partial_x v_1=\frac{1}{2}\partial_{xx}  \omega= -\frac{1}{2}\partial_{yy}{ \omega}-\frac{1}{4} \sinh(2 \omega).
	\end{align*}
	
	For the final claim, using \eqref{SinhGordonEqn} yields
	$$
	\omega_{xxxx}= - \omega_{yyxx} - \frac{1}{2}(\sinh(2 \omega))_{xx}=  \omega_{yyyy} + \frac{1}{2} (\sinh(2 \omega))_{yy} -\frac{1}{2} (\sinh(2 \omega))_{xx}.
	$$
	Using the boundary conditions implies that, on $x=0$,
	\begin{align*}
	\omega_{xxxx}&= \omega_{yyyy} - \cosh(2 \omega)  \omega_{xx} +  \cosh(2 \omega)  \omega_{yy} +2\sinh(2 \omega)  \omega_y^2\\
	&= \omega_{yyyy} + 2\cosh(2 \omega)  \omega_{yy} +2\sinh(2 \omega)  \omega_y^2+\frac{1}{2} \cosh(2 \omega) \sinh(2 \omega)\\
	&= \omega_{yyyy} + 2\cosh(2 \omega)  \omega_{yy} +2\sinh(2 \omega)  \omega_y^2+\frac{1}{4} \sinh(4 \omega).
\end{align*}
	
	Hence, we have along $x=0$,
	\begin{align*}
		\partial_x v_2 &= \frac{1}{2}  \omega_{yyyy} + \cosh(2 \omega)  \omega_{yy} +\sinh(2 \omega)  \omega_y^2+\frac{1}{8} \sinh(4 \omega)\\
		&-\frac{3}{8} \cosh (2 \omega)  \omega_{yy} -\frac{3}{16}\sinh(2 \omega) \cosh(2 \omega)-\frac{3}{4}  \omega_y^2  \omega_{yy} -\frac{3}{8}  \omega_y^2 \sinh(2 \omega)
	\end{align*}
	which verifies the final claim.
\end{proof}

We have the following restriction on $ \omega$ along $x=0$ when $\dim \mathcal{V}\leq3$.
\begin{prop}\label{InitODEProp}
	Suppose $\mathbf{F}$ parameterizes a minimal M\"{o}bius band with $\dim \mathcal{V}\leq 3$. One of the following situations occur for Jacobi functions on $\Sigma$ and  the (restriction) of the conformal factor, $\omega$: 
	\begin{enumerate}
		\item $R_{12}$ and $R_{34}$ are linearly dependent, equivalently, on the line $x=0$, $\omega=c_0$ is constant;
		\item For $b,c\in \Real$, $v_1=b R_{12}+ c R_{34}$, equivalently,  on the line $x=0$, $\omega$, satisfies the ODE 
		\begin{equation}\label{SimpleODE}
			\partial_y^2 \omega= F_1(\partial_y \omega,\omega,  b,c)= -\frac{1}{2} \sinh(2 \omega)  +b e^{-\omega}-c e^{ \omega};
		\end{equation}
		\item  For $a,b,c\in \Real$, $v_2=av_1+bR_{12}+ R_{34}$, equivalently, on the line $x=0$, $\omega$, satisfies an ODE of the form:
		\begin{equation}\label{MainODE}
			\partial_y^4 \omega= F_2( \partial_y^2  \omega, \partial_y  \omega, \omega,  a,b,c)
		\end{equation}
		where 
		\begin{align*}
			F_2(  \partial_y^2  \omega, \partial_y  \omega, \omega,a,b,c)&=-\frac{5}{4} \cosh (2 \omega)  \omega_{yy} - \frac{5}{4}  \omega_y^2 \sinh(2 \omega)+\frac{3}{2}  \omega_y^2  \omega_{yy} -\frac{1}{16}\sinh(4 \omega) \\
			&-a ( \omega_{yy} +\frac{1}{2} \sinh(2 \omega))-b  e^{-\omega}+c e^{ \omega}.
		\end{align*}
	\end{enumerate}
\end{prop}
\begin{rem} The ODE \eqref{SimpleODE} has one first integral, while the ODE \eqref{MainODE} has two.  That is, both correspond to integrable Hamiltonian systems.
\end{rem}
\begin{proof}	
	For Case (1), we note that if $R_{12}$ and $R_{34}$ are linearly dependent, then 
	$$
	b R_{12}+c R_{34}=0
	$$
	for some $b,c$ not both zero. Taking the conormal derivative along $x=0$ and appealing to Lemma \ref{ConormLem} gives
	$$
	b + c e^{2\omega}=0.
	$$
	This implies $\omega$ is constant along $x=0$.  Conversely, if $\omega=c_0$ is constant along $x=0$, then setting $b=-1$ and $c=e^{-2c_0}$ yields
	$$
	\partial_\nu (b R_{12}+c R_{34})=0,
	$$
	along $x=0$.
	By Calderon unique continuation we then have $b R_{12}+c R_{34}=0$ so $R_{12}$ and $R_{34}$ are linearly dependent.
	
	If Case (1) does not hold and $v_1, R_{12}$ and $R_{34}$ are linearly dependent, then
	$$
	v_1=b R_{12}+cR_{34}
	$$
	for some choice of $b,c\in \Real$.  Taking the conormal derivative along $x=0$ and appealing to Lemma \ref{ConormLem} gives the claimed ODE.  The converse holds by Calderon unique continuation.
	
	Finally, if neither Case (1) nor (2) hold, then $v_1$, $R_{12}$ and $R_{34}$ are linearly independent.  However, Proposition \ref{JacobiNodalProp} implies that there cannot be more than four linearly independent Jacobi functions that are invariant under $\zeta$ and that vanish on $x=0$ and $x=L_1$.
	Hence,  there must be $a,b,c$ so 
	$$
	v_2=a v_1+bR_{12}+cR_{34}.
	$$	
	Applying $\partial_x$ and using Lemma \ref{ConormLem} yields that, along $x=0$, $ \omega$ satisfies the ODE:
	\begin{align*}
		\frac{1}{2}  \omega_{yyyy} &+\frac{5}{8} \cosh (2 \omega)  \omega_{yy} + \frac{5}{8}  \omega_y^2 \sinh(2 \omega)-\frac{3}{4}  \omega_y^2  \omega_{yy} +\frac{1}{32}\sinh(4 \omega) \\
		&= -\frac{1}{2}a ( \omega_{yy} +\frac{1}{2} \sinh(2 \omega))- \frac{1}{2}be^ \omega+\frac{1}{2}c e^{-\omega}.
	\end{align*}
	
	The claim follows by manipulating the expression.  The converse is again a consequence of Calderon unique continuation.
	
\end{proof}

%
%
Finally, we note the following geometric consequence of this analysis.
\begin{prop}\label{IndPNullThm}
	If $\Sigma$ has $g(\Sigma)=1$ and satisfies
	$$
	\mathrm{ind}_D(\Sigma)+\mathrm{null}_D(\Sigma)\leq 4,
	$$
	then, $\Sigma$ contains a meridian.
\end{prop}
\begin{proof}
	We note that by the first item of Proposition \ref{InitODEProp}, $R_{12}$ and $R_{34}$ are linearly dependent if and only if $\omega$ is constant along $x=0$.  Moreover, by Proposition \ref{NullProp} in this case $\Sigma$ is, up to symmetry, the Lawson band and there is nothing further to show.
	
	 Hence, we need only rule out the possibility that $\omega$ is non-constant on $x=0$.  Propositions \ref{IndLowBndProp} and \ref{NullProp}, together imply $\mathrm{ind}_D(\Sigma)\geq 2$ and $\mathrm{null}_{D}(\Sigma)\geq 2$.  This means the main hypothesis implies $$\mathrm{ind}_D(\Sigma)=\mathrm{null}_D(\Sigma)=2.$$
	 
	 In particular, $R_{12}, R_{34}$ and $v_1$ are linearly dependent, while $R_{12}$ and $R_{34}$ are linearly independent. That is, we are in the second case of Proposition \ref{InitODEProp} and $\Omega=\omega(0, y)$  is non-constant and satisfies the ODE \eqref{SimpleODE} for appropriate choice of $b$ and $c$. 
	Now pick a $y_0\in [0, 2L_2]$ where $\Omega$ has a local maximum.  As the ODE \eqref{SimpleODE} is autonomous, $\Gamma(y)=\Omega(2y_0-y)=\omega(0, 2y_0-y)$ also solves \eqref{SimpleODE} for the same $b$ and $c$.   As $y_0$ is a local maximum, $0=\Gamma'(y_0)=\Omega'(y_0)$, while $\Gamma(y_0)=\Omega(y_0)$ by construction.  Hence, by the usual uniqueness properties of ODEs, $\Gamma(y)=\Omega(y)$.  
	
	Now let $R_{y_0}(x,y)=(x,2y_0-y)$ be reflection across the line $y=y_0$ in $\Real^2$. It is immediate from the structure of \eqref{SinhGordonEqn}, that $\omega_R=\omega\circ R_{y_0}$ solves \eqref{SinhGordonEqn}.  By the above, $\omega_R(0,y)=\omega(0,y)$.  Moreover, along $x=0$
	$$
	\partial_x \omega= \partial_x \omega_R=0
	$$
	and so by Calderon unique continuation $\omega= \omega_R$.  That is, $\omega$ is invariant under the reflection $R_{y_0}$.  This implies that $x\mapsto \mathbf{F}(x,y_0)$ parameterizes a geodesic in $\mathbb{S}^3$, in particular, $\Sigma$ contains a meridian. 
\end{proof}

\section{Index of the Lawson band}
We rigorously compute the index and nullity of $\Sigma_{LB}$.

\begin{prop} \label{LBIndNullProp}The generalized Lawson bands $\bar{\tau}_{1,2k}$ satisfy
	$$
 \mathrm{null}_{D}(\bar{\tau}_{1,2k})\in \set{1,3}, 	\mathrm{ind}_{D}(\bar{\tau}_{1,2k})+\mathrm{null}_{D}(\bar{\tau}_{1,2k})\leq 1+  2\lceil
	\frac{1}{2}\sqrt{8k^2+1} \rceil.
	$$
Moreover, $\bar{\tau}_{1,2}=\Sigma_{LB}$ satisfies
	$$\mathrm{ind}_D(\Sigma_{LB})=2\mbox{  and } \mathrm{null}_D(\Sigma_{LB})=1.
	$$
	We may also deduce, using computer assistance, that
	$$
   \mathrm \mathrm{ind}_D(\bar{\tau}_{1,4}) =4  \mbox{ and } \mathrm{null}_D(\bar{\tau}_{1,4})= 1.
	$$
Finally, let $\lambda_D^1(\bar{\tau}_{1,2k})$ denote the first Dirichlet eigenvalue of the stability operator on $\bar{\tau}_{1,2k}$.  It follows that
$$
\lambda_D^1(\bar{\tau}_{1,2k})=-2.
$$
\end{prop}

\begin{rem}
    We note that Morozov-Penskoi \cite{morozovIndexMinimalSurfaces2023} computed the Morse index of $\tau_{1,2}$, i.e., the full Lawson Klein bottle, to be $7$ and it's nullity to be $5$.
\end{rem}
\begin{proof}
	Let us parameterize the oriented double cover, $\bar{\tau}_{1,2k}'$ of $\bar{\tau}_{1,2k}$ by 
	$$
	\mathbf{F}_k: [0,\pi]\times [0, 2\pi)\to \bar{\tau}_{1,2k}' \subset \mathbb{S}^3
	$$
	where 
	$$
	\mathbf{F}_k(x,y)= (\cos y\cos x, \sin y \cos x, \cos 2k y \sin x, \sin 2k y \sin x).
	$$
	In the coordinates $x$ and $y$ we may express the stability operator of $\bar{\tau}_{1,2k}'$ as
	$$
	L_{\bar{\tau}_{1,2k}'} f= \partial_x^2 f+ a_k(x) \partial_x f + b_k(x) \partial_y^2 f + V_k(x) f
	$$
	where
	$$
	a_k(x)=\frac{\left(4k ^2-1\right) \sin (2 x)}{4k ^2+1-\left(4k ^2-1\right) \cos (2 x)}
	$$
	$$
	b_k(x)=\frac{2}{4 k^2+1-\left(4 k^2-1\right) \cos (2 x)}
	$$
	and
	$$
	V_k(x)=2 + \frac{32k^2}{\left(4k^2+1-\left(4k^2-1\right) \cos (2 x)\right)^2}.
	$$
	In particular, we expect to find a basis of odd eigenfunctions of the eigenspace of $-L_{\bar{\tau}_{1,2k}'}$ associated to $\lambda$ of the form
	$$
	f(x,y)=f_n(x) \cos ny \mbox{ or }  f(x,y)=g_n(x) \sin n y
	$$
	where $f_n$ (and $g_n$) is even in the sense that $f_n(\pi-x)=f_n(x)$ when $n$ is odd and odd in the sense that $f_n(\pi-x)=-f_n(x)$ when $n$ is even.  Moreover, $f_n$ (and  $g_n$) satisfy
	$$
	\partial_x^2 u+ a_k(x) \partial_x u -n^2  b_k(x)  u+ V_k(x) u+\lambda u=0.
	$$
	One may verify that this is a Sturm-Liouville eigenvalue problem associated to the energy
	$$
	E_n^k[u]= \int_{0}^\pi \left( |u'|^2 +(n^2 b_k(x)- V_k(x)) u^2 \right) w_k(x)dx
	$$
	with Dirichlet boundary conditions  where here
	$$
	w_k(x)=\sqrt{4k^2+1-\left(4 k^2-1\right) \cos (2 x)}.
	$$
	
	One readily computes that for $\bar{\tau}_{1,2k}'$ one has
	$$R_{12}=  \frac{\sqrt{2} k\sin (2 x)}{\sqrt{4k^2+1-\left(4k^2-1\right) \cos (2 x)}}=f_0(x)
	$$
	This gives a non-trivial element in the null space.
	We note that $f_0$ has exactly two nodal domains, $(0, \frac{\pi}{2})$ and $(\frac{\pi}{2}, \pi)$. It readily follows from Sturm-Liouville theory and the fact that $b_k(x)\geq 0$, that any function $u$ that is odd and any $n\geq 0$ even 
	$$
	E_n^k[u]\geq E_0^k[f_0] = 0
	$$
	with strict equality unless $n=0$ and $u$ is a constant multiple of $f_0$.  It follows that any eigenfunction of the form $f_{2n}(x) \cos (2n y) $ for $n>0$ cannot contribute to either the index or nullity.
	
	For $n$ odd, we note that if $f_n$ is a non-trivial even eigenfunction associated to $E_n$ that changes sign, then it cannot have a zero at $x=\pi/2$.  This is because evenness implies $f_n'(\frac{\pi}{2})=0$.  Moreover, if $f_n$ has any zero in $(0,\frac{\pi}{2})$ then the associated odd eigenfunction $f_n(x) \cos n y$ to $L_{\bar{\tau}_{1,2k}'}$ has a nodal domain contained inside $[0,\frac{\pi}{2}]\times [0, 2\pi]$ and so by domain monotonicity cannot contribute to the index or nullity.
	From this we conclude that the only elements of the form $f_n(x) \cos ny $ that can contribute to the index or nullity is of the form $f_n(x) \cos ny $ where $n$ is odd and $f_n$ does not change sign in $(0,\pi)$, i.e., is the eigenfunction associated to the lowest eigenvalue of $E_n$.

	Now we compute that
	$$
	n_2=\frac{2 \sqrt{2} k \sin (x) \cos (y)}{\sqrt{-\left(4 k^2-1\right) \cos (2 x)+4 k^2+1}}=f_1(x) \cos y
	$$
	which is a non-trivial element contributing to the index. We observe that $f_1$ does not change sign in $(0,\pi)$ so is the lowest eigenfunction associated to $E_1$.   An immediate consequence is to verify that
	$$
	\lambda_D^1(\bar{\tau}_{1,2k})=-2.
	$$
	Furthermore, as $b_k(x)\geq 0$, it follows that for any odd $n$ and any even function, $u$. one has
	$$
	E_n^k[u]=\int_0^\pi ((n^2-1) b_k(x) -2) u^2(x) w_k(x) dx + E_1^k[u]+ 2\int_0^\pi u^2 w_k(x) dx
	$$
	$$
	\geq  \int_0^\pi ((n^2-1)  b_k(x) -2)u^2(x) w_k(x) dx.
	$$
	Here we used the variational characterization of $f_1$ as the lowest eigenfunction of $E_1$ to see that
	$$
	 E_1^k[u]+ 2\int_0^\pi u^2 w_k(x) dx\geq  E_1^k[f_1]+ 2\int_0^\pi f_1^2 w_k(x) dx=0.
	 $$
	One readily checks that $b_k$ has a strict local minimum at $x=\frac{\pi}{2}$ where it takes on the value $\frac{1}{4k^2}$.  It follows that either $n=1$ and $u=f_1$ or
	$$
	E_n^k[u]>\frac{1}{4k^2}\int_0^\pi ((n^2-1-8k^2 )u^2(x) w_k(x) dx.
	$$
	It follows that when $k=1$, for any odd $n\geq 3$ one has
	$E_n^k[u]>0$ and so $\mathcal{U}'_-=\mathrm{span}(n_1, n_2)$ and $\mathcal{D}_-'=\mathrm{span}(R_{12})$.  That is, 
	$$\mathrm{ind}_D(\Sigma_{LB})=2\mbox{  and } \mathrm{null}_D(\Sigma_{LB})=1
	$$
	as claimed. 
	For $k>1$, one can only obtain contributions to the index or nullity from odd $n$ with $n<\sqrt{8k^2+1} $.  Note that by monotonicity of domain the largest such odd $n$ can contribute either contribute zero or two to the nullity and so the nullity is either one or three. Counting the number of such odd $n$ satisfying this inequality and noting they lead to two linearly independent elements gives the claimed estimate.

  Specializing to $\bar{\tau}_{1,4}$, i.e., $k=2$ one observes that the only possible odd $n$ that can contribute to the index or nullity are $n=1,3,5$. Let 
  $$
  L_5^2= w_2^{-1}(x) \partial_x (w_2(x) \partial_x)+(V_2(x)-25 b_2(x))
  $$
  be the operator associated to $E_5^2$.   Set $f(x)=\sin(x)^{\frac{5}{8}}$ an involved computation (carried out by the authors in Mathematica) allows one to verify that
  $$
  L_5^2 f<0,
  $$
  that is $f$ is a supersolution of $L_5^2$.  It follows that  
   for any $u$ of compact support in $[0,\pi]$
  $$
  E_5^2[u]> 0.
  $$
  That is, $n=5$ cannot contribute to the index or nullity.  In particular,
  $$
  \mathrm{ind}_D(\bar{\tau}_{1,4})+\mathrm{null}_D(\bar{\tau}_{1,4})\leq 5.
  $$ 
  Finally, as $e(\bar{\tau}_{1,4})=4$, Proposition \ref{NodalCountWProp} implies the index is strictly greater than two and thus equals four from which the claim follows.
\end{proof}

We note that the computations of \cite[Section 6]{simonsMinimalVarietiesRiemannian1968} can be readily adapted to this setting which gives the following:
\begin{prop}
	Let $\Sigma$ be a non-orientably minimal surface bounded by $C$.  If 
	$
	\lambda_{D}^1(\Sigma)
	$, the first Dirichlet eigenvalue of the stability operator of $\Sigma$, satisfies
 $\lambda_{D}^1(\Sigma)<-3$, then $C(\Sigma)$, the cone over $\Sigma$ is unstable.  Conversely, if
 $\lambda_{D}^1(\Sigma)>-3$, then $C(\Sigma)$ is stable.  In particular, $C(\bar{\tau}_{1,2k})$ is stable.
\end{prop}
During the final preparation of this manuscript Parise-Guaraco \cite{guaracoWhitesConeMobius2026} showed that the cone over $\bar{\tau}_{1,2}$ is actually area-minimizing.

\section{Problems}

In \cite{savoIndexBoundsMinimal2010}, building on work of Ros \cite{rosOnesidedCompleteStable2006}, it was shown that the index of oriented minimal hypersurfaces in the sphere grows linearly in the first Betti number.  The proof of Proposition \ref{NodalCountWProp} gives a growth bound in terms of both the genus and Euler number.
\begin{ques}
  Are there quantitative growth estimates on the index of a minimal surface spanning $C$ only in terms of the genus?  Only in terms of the Euler number?
\end{ques}

The index and nullity of the Lawson minimal surfaces $\xi_{k,1}$ are computed in \cite{kapouleasIndexNullityLawson2020}.
\begin{ques}
  What are the Morse indices of the twisted Lawson surfaces $\tilde{\xi}_{m,k}$?  To what extent does this characterize these surfaces?
\end{ques}
A complete analysis should allow one to answer the following:
\begin{conj}
The cones over the $\tilde{\xi}_{m,l}$ are all unstable in $\Real^4$.
\end{conj}
More ambitiously:
\begin{conj}
If $\Sigma$ is an embedded non-orientable minimal surface spanning $C$, then its cone is stable if and only if $\Sigma$ is a Lawson band.
\end{conj}

Recall that $\bar{\tau}_{1,4}$ has index four: 
\begin{ques}
   Is there an index three (immersed) minimal surface spanning $C$?  What is the next index that occurs for embedded surfaces?  
\end{ques}

\end{document}